\documentclass[11pt,reqno]{amsart}
\usepackage{latexsym}
\usepackage{mathrsfs}
\usepackage{multirow}
\usepackage{multicol}
\usepackage{amssymb, amsmath}
\usepackage{graphicx}
\usepackage{color}
\usepackage{bbm}
\usepackage{enumerate}
\usepackage{amsmath,bm}
\allowdisplaybreaks[4]
\usepackage[numbers,sort&compress]{natbib}
\usepackage{multirow}
\numberwithin{equation}{section}

\newcommand{\be}{\begin{equation}}
\newcommand{\ee}{\end{equation}}

\newcommand{\ben}{\begin{eqnarray*}}
\newcommand{\enn}{\end{eqnarray*}}

\newcommand{\beq}{\begin{equation}}
\newcommand{\eeq}{\end{equation}}
\newcommand{\bea} {\begin{array}{rl}}
\newcommand{\eea} {\end{array}}
\newcommand{\bepa}{\left\{ \begin{array}{l}}
\newcommand{\eepa} {\end{array}\right.}

\newcommand{\Rmnum}[1]{\expandafter\@slowromancap\romannumeral #1@}

\newtheorem{theorem}{\textbf Theorem}[section]
\newtheorem{lemma}{\textbf Lemma}[section]
\newtheorem{remark}{\textbf Remark}[section]

\newtheorem{proposition}[theorem]{Proposition}
\def\dis{\displaystyle}
\begin{document}

\author{Huicong Li}
\address{School of Mathematics, Sun Yat-Sen University, Guangzhou, 510275, Guangdong Province, China.}
\email{hli7@tulane.edu}

\author{Rui Peng}
\address{School of Mathematics and Statistics, Jiangsu Normal University, Xuzhou, 221116, Jiangsu Province, China}
\email{pengrui\_seu@163.com}

\author{Zhi-an Wang}
\address{Department of Applied Mathematics, Hong Kong Polytechnic University,
Hung Hom, Kowloon, Hong Kong}
\email{mawza@polyu.edu.hk}
\title
[On a Diffusive SIS Epidemic Model]
{On a Diffusive SIS Epidemic Model with Mass Action Mechanism and birth-death effect: analysis, simulations and comparison with other mechanisms$^*$}\thanks{\footnotesize$^*$H. Li was partially supported by NSF of China (Nos. 11601170, 11671143 and 11671144), and
R. Peng was partially supported by NSF of China (Nos. 11671175 and 11571200), the Priority Academic Program Development of Jiangsu Higher Education Institutions,
Top-notch Academic Programs Project of Jiangsu Higher Education Institutions (No. PPZY2015A013) and Qing Lan Project of Jiangsu Province, and
Z.-A. Wang was partially supported by the Hong Kong RGC GRF grant (No. PolyU 153032/15P)}

\begin{abstract} In the present paper, we are concerned
with an SIS epidemic reaction-diffusion model governed
by mass action infection mechanism and linear birth-death growth with no flux boundary condition.
By performing qualitative analysis, we study the stability of the disease-free equilibrium, uniform persistence property in terms of the basic reproduction number and the global stability of the endemic equilibrium in homogeneous environment, and investigate the
asymptotic profile of endemic equilibria (when exist) in heterogeneous environment as one of the movement rate of the susceptible and infected populations is small. Our results, together with those in previous works on three other closely related modeling systems,
suggest that the factors such as infection mechanism, variation of total population and population movement play vital but subtle roles in the transmission dynamics of diseases and hence provide useful insights into the strategies designed for disease control and prevention.
\end{abstract}


\date{}
\maketitle

\section{Introduction}
The mathematical study of infectious diseases can be traced back to the classic work of Kermack and McKendrick \cite{KM26} in 1927.
In \cite{KM26}, the authors adopted the {\it mass action infection mechanism} (also called {\it density-dependent infection mechanism}) to study a deterministic SIR (susceptible-infected-recovered)
epidemic model, meaning that the infection (incidence) rate is proportional to the number of encounters between susceptible and infected individuals; mathematically,
such infection rate is characterized by the bilinear function $\beta SI$, where $\beta>0$ is the disease transmission rate and $S(t)$ and $I(t)$ represent the density of susceptible and infected
populations respectively. The most significant achievement made in \cite{KM26} is perhaps the epidemic threshold result that the density of
susceptible individuals must exceed a critical value in order for the epidemic outbreak to occur. Due to the seminal importance of
the Kermack-McKendrick theory to the field of theoretical epidemiology, their works were republished in 1991; see \cite{KM911, KM912, KM913}.

Employing the same infection mechanism and instead considering an SIS (susceptible-infected-susceptible) model,
one is led to the following ODE system (see, for instance, \cite{Martcheva}):
 \begin{equation}
\left\{ \begin{array}{ll}
S' = -\beta SI + \gamma I, &t>0,\\
I' = \beta SI - \gamma I,&t>0,
\end{array}\right.
\label{intro1}
\end{equation}
where $\gamma>0$ is the disease recovery rate, together with initial data fulfilling $S(0)+I(0)=N>0$ and $I(0)>0$.
As one of the simplest models in mathematical epidemiology, \eqref{intro1} still demonstrates the threshold result as
Kermack and McKendrick \cite{KM26} observed. In fact, it is clear that
$S(t) + I(t) = N$ for all $t\geq 0$,
and hence \eqref{intro1} can be reduced to the following logistic type equation:
$$I' = \beta I \left[\left(N-\frac{\gamma}{\beta}\right) -I \right].$$
Simple analysis shows that if $N\leq \gamma/\beta$, then $I(t) \to 0$ and in turn $S(t) = N - I(t) \to N$ as
$t \to \infty$, while if $N > \gamma/\beta$, it holds $I(t) \to N - \gamma/\beta >0$, and $S(t) \to \gamma/\beta>0$ as $t\to \infty$.
Defining the {\it basic reproduction number} $\mathcal R_0= N\beta / \gamma$, then the {\it disease-free equilibrium} (DFE) $(N,0)$ is globally attractive if $\mathcal R_0 \leq 1$, while the {\it endemic equilibrium} (EE) $(\gamma / \beta, N - \gamma/\beta)$ is globally attractive if $\mathcal R_0>1$.
We also refer interested readers to the review paper \cite{Het2} for various ODE models describing infectious diseases.

Nowadays it is widely recognized that spatial spread of an infection is closely related to the heterogeneity
of the environment and the spatial-temporal movement of the hosts. This is well supported by numerous research on diseases including
malaria \cite{Lou-Zhao1, Lou-Zhao2}, rabies \cite{Kallen1, Kallen2, Murray1}, dengue fever \cite{Takahashi}, West Nile virus \cite{Lewis, Lin}, hantavirus \cite{Abramson1, Abramson2}, Asian longhorned beetle \cite{GL,GZ}, etc; see \cite{Ruan} and references therein.
A popular way to incorporate spatial movement of hosts into epidemic
models is to assume host random movements, leading to coupled reaction-diffusion equations.
Taking into account spatial diffusion and environmental heterogeneity, we obtain the PDE version of \eqref{intro1}:
\begin{equation}
\left\{ \begin{array}{llll}
S_t - d_S \Delta S = -\beta(x) SI +\gamma(x) I, &x\in\Omega,\, t>0,\\
I_t - d_I \Delta I = \beta(x) SI -\gamma(x) I,& x\in\Omega,\,t>0,\\
\frac{\partial S}{\partial \nu}= \frac{\partial I}{\partial \nu} = 0,&x\in\partial\Omega, \, t>0,
\end{array}\right.
\label{intro2}
\end{equation}
where the spatial domain $\Omega \subset \mathbb R^m~{\rm(}m\geq 1{\rm)}$ is bounded and has smooth boundary $\partial\Omega$; positive constants $d_S$ and $d_I$ represent the diffusion rate of susceptible and infected individuals respectively; $\beta(x)$ and $\gamma(x)$ are
positive H\"{o}lder continuous functions on $\overline \Omega$ accounting for the disease transmission rate and recovery rate respectively; the Neumann boundary condition means that no population flux crosses the boundary $\partial\Omega.$ For this model, Deng and Wu \cite{Deng-Wu} studied the global dynamics and existence of EE, while \cite{LiBo1,Wu-Zou} investigated the asymptotic profile of EE (when exists) as the diffusion rate of susceptible or infected population is small or large, which consequently suggests interesting implication in terms of epidemiology; see the last section of our paper for further discussion.

System \eqref{intro2} does not take into consideration birth/death effect of susceptible or infected individuals and thus the total population is conserved in the sense that
$$\int_\Omega \left[S(x,t)+I(x,t)\right]dx = \int_{\Omega} \left[S_0(x) + I_0(x) \right]dx =: N ,\quad \forall t\geq 0.$$
However, it is quite natural to consider the situation that susceptible individuals are subject to a recruitment (source) term modeling their birth and death rate,
especially a linear one \cite{AM, Het2}. Therefore, in this paper we are motivated to study the following reaction-diffusion epidemic system with varying total population and environmental heterogeneity:
 \begin{equation}
 \left\{ \begin{array}{llll}
  S_t-d_S\Delta S=\Lambda(x)-S-\beta(x)SI+\gamma(x)I,&x\in\Omega,\,t>0,\\
  I_t-d_I\Delta I=\beta(x)SI-\left[\gamma(x)+\mu (x)\right] I,&x\in\Omega,\,t>0,\\
  \frac{\partial S}{\partial \nu}=\frac{\partial I}{\partial \nu}=0,&x\in\partial\Omega,\,t>0,\\
  S(x,0)=S_0(x)\geq0,\,I(x,0)= I_0(x)\geq,\not\equiv 0, &x\in\Omega.
 \end{array}\right.
 \label{model}
 \end{equation}
The recruitment term $\Lambda(x)-S$ represents that the susceptible population is subject to linear growth and
$\mu(x)$ accounts for the death rate of the infected, with $\Lambda$ and $\mu$ being assumed to be positive H\"{o}lder functions on $\overline \Omega$. All the other parameters have the same interpretation as before. Throughout the paper, the initial data $S_0$ and $I_0$ are nonnegative continuous functions on $\overline\Omega$,
and there is a positive number of infected individuals initially, i.e., $\int_{\Omega}I_0(x)dx>0$.

Another widely accepted type of infection mechanism is the so-called {\it frequency-dependent transmission} (also called as {\it standard incidence infection mechanism}) of the form $\beta SI/(S+I)$, initiated by de Jong, Diekmann and Heesterbeek \cite{deJong} in 1995. In this scenario, \eqref{intro2} becomes
\begin{equation}
\left\{ \begin{array}{llll}
\dis S_t - d_S \Delta S = -\beta(x) \frac{SI}{S+I} +\gamma(x) I, &x\in\Omega,\, t>0,\\
\dis I_t - d_I \Delta I = \beta(x) \frac{SI}{S+I} -\gamma(x) I,& x\in\Omega,\,t>0,\\
\dis \frac{\partial S}{\partial \nu}= \frac{\partial I}{\partial \nu} = 0,&x\in\partial\Omega, \, t>0,\\
S(x,0)=S_0(x),\, I(x,0)=I_0(x),&x\in\Omega,
\end{array}\right.
\label{intro3}
\end{equation}
and its counterpart with linear recruitment reads
\begin{equation}
\left\{ \begin{array}{llll}
\dis S_t - d_S \Delta S = \Lambda(x) - S -\beta(x) \frac{SI}{S+I} +\gamma(x) I, &x\in\Omega,\, t>0,\\
\dis I_t - d_I \Delta I = \beta(x) \frac{SI}{S+I} -\gamma(x) I,& x\in\Omega,\,t>0,\\
\dis \frac{\partial S}{\partial \nu}= \frac{\partial I}{\partial \nu} = 0,&x\in\partial\Omega, \, t>0,\\
S(x,0)=S_0(x),\, I(x,0)=I_0(x),&x\in\Omega.
\end{array}\right.
\label{intro4}
\end{equation}
We note that \eqref{intro3} was first proposed by Allen et al. \cite{Allen} and then it (and its variants) was (were) studied extensively by many researchers \cite{Cui-Lou, Cui-Lam-Lou,DHK,FLW,Ge, Ge-Lin-Zhang, HHL, Kousuke,Peng, Peng-Liu,Peng-Yi,PZ-non-12} while \eqref{intro4} was analyzed by Li et al. \cite{LiPengWang}; see also \cite{LiBo2} for the case of logistic source instead of the linear one. One also observes that the total population in \eqref{intro3} is conserved and that in \eqref{intro4} varies.

In \cite{McCallum}, by comparing the outcomes of models with density-dependent
and frequency-dependent transmission rates to the observed epidemiology of certain diseases,
McCallum, Barlow and Hone concluded that both density-dependent
and frequency-dependent mechanisms have their own advantages in modeling disease spread,
depending on the transmission mode of the disease under consideration. They further pointed out that the transmission mode could be in general decided by estimating the force of infection.

On the other hand, epidemic theory for many ODE models has demonstrated that {\it the basic reproduction
number}, which may be considered as the fitness of a pathogen in a given population,
must be greater than unity for the pathogen to invade a susceptible population; see \cite{AM2,BCC,DH,JS,Murray,Th} and references therein. For the PDE models \eqref{intro2}-\eqref{intro4}, we can also find their respective {\it basic reproduction number} $\mathcal R_0$ and show that $\mathcal R_0$ serves as the threshold value to determine the transmission dynamics of disease; that is, if $\mathcal R_0>1$ the disease persists whereas it becomes extinct in the long run if $\mathcal R_0<1$.
However, the total population $N$, and the movement (migration) rates $d_I$ and $d_S$ may affect $\mathcal R_0$ of \eqref{intro2}-\eqref{intro4} in different manners.
As a result, each of the parameters $N,\,d_I,\,d_S$ plays a subtle role in  disease control; more detailed description will be made in the last discussion section.

The main goal of the current paper is twofold. The first one is to rigorously investigate qualitative properties of
\eqref{model} and the asymptotic profile of EE (when exists) with respect to the small movement rate $d_I$ or $d_S$.
Theorem \ref{uni-persis} below tells us that once $\mathcal R_0>1$, the infectious disease will uniformly
persist in space. Thus it becomes important to understand how the mobility of population migration
affects the spatial distribution of disease, because this will help decision-makers
to predict the pattern of disease occurrence and henceforth to conduct effective/optimal control strategies of disease eradication. Our result in Theorem \ref{model-ss-ds} indicates that restricting the motility rate of susceptible individuals cannot eradicate the disease for \eqref{model}, while this strategy works perfectly for \eqref{intro2} with small total population size (\cite[Corollary 2.4]{Wu-Zou}). Similar phenomenon was also observed in models \eqref{intro3} and \eqref{intro4}. Therefore, this suggests that varying total population tends to enhance the persistence of infectious disease.
The second goal is to compare our main results on the model \eqref{model} with those on models  \eqref{intro2}, \eqref{intro3} and \eqref{intro4}, so as to understand the influence of the factors such as infection mechanism, movement rate and source term on the eradication of epidemics, and to discuss possible applications in disease control. Numerical simulations are also carried out to reinforce the theoretical findings and illustrate possible outcomes for those unknown situations and hence provide clues for further analytical pursues. We refer to Section \ref{dis-sec} for detailed discussion on the implications of analytical results and comparisons between four related SIS epidemic models mentioned above.

The remainder of this paper is organized as follows. In Section \ref{sec-2}, we first obtain the global existence and boundedness of solutions to the parabolic problem \eqref{model}, then discuss the stability of equilibrium and the uniform persistence property via the basic reproduction number $\mathcal R_0$, and finally we consider the global attractivity of DFE and EE in spatially homogeneous environment. Section \ref{sec-3} is devoted to the study of asymptotic profile of EE when the diffusion rate of susceptible population or infected population approaches zero. In the last section, we perform numerical simulations, compare our results for \eqref{model} with those of the other three models, and discuss the implication of our findings in detail from the viewpoint of disease control.

In the rest of the paper, for notational convenience, we denote
 \begin{equation}
 g^*=\max_{x\in\overline\Omega}g(x) \ \ \mbox{and}\ \  g_*=\min_{x\in\overline\Omega}g(x),\ \
 \mbox{for $g=\Lambda,\,\beta,\,\gamma$ and $\mu$}.
 \nonumber
 \end{equation}

\section{Properties of solutions to \eqref{model}}
\label{sec-2}
In this section, we consider the parabolic system \eqref{model}
by first establishing the global existence and uniform boundedness of solutions,
and then show the local stability of DFE and uniform persistence via the basic reproduction number.
Lastly we investigate the global attractivity of DFE and EE in homogeneous environment.

\subsection{Global existence and uniform boundedness}

We now establish the global existence and boundedness of solutions to \eqref{model}.

\begin{theorem}\label{global-bdd}
The solution $(S(x,t),I(x,t))$ of problem \eqref{model} exists uniquely and globally.
Furthermore, there exists a positive constant $M$ depending on initial data and the parameters
$d_S$, $d_I$, $\Lambda$, $\beta$, $\gamma$ and $\mu$ such that
 \begin{equation}
\label{global-bdd-eq}
 \|S(\cdot,t)\|_{L^{\infty}(\Omega)}+\|I(\cdot,t)\|_{L^{\infty}(\Omega)}\leq M,\quad \forall t\geq 0.
 \end{equation}
Moreover, there exists some $M'>0$ independent of initial data fulfilling
\begin{equation}
 \|S(\cdot,t)\|_{L^{\infty}(\Omega)}+\|I(\cdot,t)\|_{L^{\infty}(\Omega)}\leq M',\quad \forall t\geq T,
\label{global-bdd-eq2}
\end{equation}
for some large $T>0$.
 \end{theorem}

\begin{proof} From the standard theory for semilinear parabolic systems \cite{Amm}, it follows that \eqref{model} admits
a unique solution $(S(x,t),I(x,t))$ for $x\in\overline\Omega$ and $t\in [0,T_{\rm max})$ with $T_{\rm max}$ being
the maximal existence time. Moreover, the strong maximum principle for parabolic equations yields that the solution is positive on $\overline\Omega\times (0,T_{\rm max})$.
Integrating both PDEs of \eqref{model} and adding the resulting two identities, we are led to
\begin{align}
\frac{d}{dt}\int_\Omega (S(x,t)+I(x,t))dx
&=\int_\Omega \Lambda(x)dx-\int_\Omega (S(x,t)+\mu(x)I(x,t))dx \nonumber\\
&\leq \int_\Omega \Lambda(x)dx- \theta \int_\Omega (S(x,t)+I(x,t))dx,
\label{S-Gron}
\end{align}
where $\theta=\min\{1,\mu_* \}>0$. Then the well-known Gronwall's inequality applied to \eqref{S-Gron} asserts that there exists some constant $M_1>0$, such that
\begin{equation}
\int_\Omega (S(x,t)+I(x,t))dx \leq M_1,\quad \forall t\in (0,T_{\rm max}).
\label{S-I-L1}
\end{equation}

We now consider
\begin{equation}
\left\{\begin{array}{lll}
 S_t-d_S \Delta S=\Lambda(x)-S+\left[\gamma(x)-\beta(x)S\right]I, &x\in\Omega,\, t\in (0,T_{\rm max}),\\
 \frac{\partial S}{\partial \nu}=0,&x\in\partial \Omega,\, t\in (0,T_{\rm max}),\\
 S(x,0)=S_0(x),&x\in\Omega.
\end{array}\right.
\label{S-prob}
\end{equation}
For any nonnegative $I$, it is straightforward to verify that the positive constant
 $$
 M_2:=\max\left\{\|\Lambda\|_{L^\infty(\Omega)},\ \|S_0\|_{L^\infty(\Omega)},\ \left\| \frac{\gamma}{\beta}\right\|_{L^\infty(\Omega)} \right\}
 $$
is an upper solution of \eqref{S-prob}. The comparison principle for parabolic equations gives
 $$
 S(x,t)\leq M_2,\quad \forall x\in\overline\Omega,\ t\in (0,T_{\rm max}).
 $$
Since $S$ is uniformly bounded and the $L^1$-norm of $I(\cdot,t)$ is also bounded for $t\in (0,T_{\rm max})$ thanks to \eqref{S-I-L1},
in view of \cite[Theorem 3.1]{Alikakos} or \cite[Lemma 3.1]{PZ-non-12} and using the $I$-equation, we deduce that $I$ is also uniformly bounded
in $\overline\Omega\times (0,T_{\rm max})$. As a result, we must have $T_{\rm max}=\infty$ and \eqref{global-bdd-eq} is proved.

We next show \eqref{global-bdd-eq2}. To the aim, we need to construct a more accurate upper solution of problem \eqref{S-prob}, which is independent of $S_0$ for all large time.
In fact, let $u(t)$ be the unique solution of the following ODE.
\begin{equation}
u'(t) = \Lambda^* + \|\gamma/\beta\|_{L^\infty(\Omega)} - u(t),\quad t>0;\ \
u(0) = \|S_0\|_{L^\infty(\Omega)} + \|\gamma/\beta\|_{L^\infty(\Omega)}.
\nonumber
\end{equation}
It is clear that
\begin{equation}
u(t) = \left(\|S_0 \|_{L^\infty(\Omega)} + \left\|\frac{\gamma}{\beta}\right\|_{L^\infty(\Omega)} \right) e^{-t} + \left(\Lambda^* +\left\|\frac{\gamma}{\beta}\right\|_{L^\infty(\Omega)} \right) \left(1 - e^{-t}\right) \geq \left\|\frac{\gamma}{\beta}\right\|_{L^\infty(\Omega)},
\nonumber
\end{equation}
which implies $\gamma(x)-\beta(x) u(t) \leq 0,\ \forall x\in\overline\Omega,\, t>0.$
It can be easily checked that $u(t)$ is an upper solution of \eqref{S-prob} and consequently,
 $$
 S(x,t) \leq u(t) \to \Lambda^* +\left\|\frac{\gamma}{\beta}\right\|_{L^\infty(\Omega)}\quad \mbox{as }t\to \infty,\quad \forall x\in\overline\Omega.
 $$
That is, we obtain an upper bound of $\|S(\cdot,t)\|_{L^\infty(\Omega)}$ which is independent of initial data for all large time. Now applying \cite[Lemma 3.1]{PZ-non-12} to the $I$-equation, we deduce that $\|I(\cdot,t)\|_{L^\infty(\Omega)}$ can also be bounded by a positive constant independent of $(S_0, I_0)$ for large $t>0$.
\end{proof}

\subsection{Basic reproduction number and uniform persistence}

It is easily seen that the following elliptic problem
 \begin{equation}
 \label{dfe}
 -d_S\Delta S=\Lambda(x)-S,\ x\in\Omega;\ \
 \frac{\partial S}{\partial \nu}=0,\ x\in\partial\Omega
 \end{equation}
admits a unique positive solution $\tilde S$, which is globally asymptotically stable for the
corresponding parabolic equation with nonnegative initial data. Then $(\tilde S,0)$ is an equilibrium of \eqref{model},
which we call the disease-free equilibrium (DFE). Clearly, it is the unique DFE.

We define the basic reproduction number $\mathcal{R}_0$ as follows:
 \begin{equation}
 \mathcal{R}_0=\sup_{0\neq \varphi\in H^1(\Omega)}\frac{\int_{\Omega}\beta\tilde S \varphi^2 dx}{\int_{\Omega}\left[d_I|\nabla\varphi|^2+(\gamma+\mu) \varphi^2\right] dx}.
 \label{r0}
 \end{equation}
Indeed, one can follow the idea of next generation operators as in \cite{PZ-non-12} to introduce the basic reproduction number,
which coincides with the value $\mathcal{R}_0$. It is worth mentioning that the basic reproduction number $\mathcal{R}_0$ defined here is
qualitatively different from that in \cite{Allen} and \cite{Deng-Wu} in that it also depends implicitly on the diffusion rate $d_S$ of the susceptible individuals.

Let $\left(\lambda^*,\psi^*\right)$ be the principal eigenpair of the eigenvalue problem
 \begin{equation}
 d_I\Delta u+(\beta \tilde S-\gamma-\mu)u+\lambda u=0,\ x\in\Omega;\ \
 \frac{\partial u}{\partial\nu}=0,\ x\in\partial\Omega.
 \label{eigen-prob}
 \end{equation}
Then, we have the following properties of $\mathcal{R}_0$, the proof of which resembles that of \cite[Lemma 2.3]{Allen} and hence is omitted.

\begin{proposition}\label{property} The following assertions hold.
\begin{enumerate}
 \item[\rm{(a)}] $\mathcal{R}_0$ is a monotone decreasing function of $d_I$
with $\mathcal{R}_0\rightarrow \max_{\overline\Omega} \beta \tilde S/(\gamma+\mu)$ as $d_I\rightarrow0$ and
$\mathcal{R}_0\rightarrow\int_\Omega\beta \tilde S \mbox{d}x/{\int_\Omega(\gamma+\mu)\mbox{d}x}$ as
$d_I\rightarrow\infty$;
 \item[\rm{(b)}] If $\int_\Omega \beta(x) \tilde S(x)\mbox{d}x<\int_\Omega [\gamma(x)+\mu(x)]\mbox{d}x$, and $\beta \tilde S - (\gamma+\mu)$ changes sign, then there exists a
threshold value $d_I^*\in(0,\infty)$ such that $\mathcal{R}_0>1$ for
$d_I<d_I^*$ and $\mathcal{R}_0<1$ for $d_I>d_I^*$;
 \item[\rm{(c)}] If $\int_\Omega \beta(x) \tilde S(x)\mbox{d}x > \int_\Omega [\gamma(x)+\mu(x)]\mbox{d}x$, then $\mathcal{R}_0>1$ for all $d_I>0$.
 \item[\rm{(d)}] $\mathcal{R}_0>1$ when $\lambda^*<0$, $\mathcal{R}_0=1$ when $\lambda^*=0$, and $\mathcal{R}_0<1$ when $\lambda^*>0$.
\end{enumerate}
\end{proposition}

It turns out that the stability of the DFE $(\tilde S,0)$ is completely determined by the size of $\mathcal{R}_0$.

\begin{proposition}\label{locstab}
The DFE $(\tilde S,0)$ is linearly stable if $\mathcal{R}_0<1$, and it is linearly unstable if $\mathcal{R}_0>1.$
\end{proposition}

\begin{proof}
The linearization of \eqref{model} around the DFE $(\tilde S,0)$ reads
 \begin{equation}
 \left\{ \begin{array}{lll}
  \eta_t-d_S\Delta\eta =-\eta+(-\beta\tilde S+\gamma)\xi,&x\in\Omega,t>0,\\
  \xi_t -d_I\Delta \xi =(\beta\tilde S-\gamma-\mu)\xi, &x\in\Omega,t>0,\\
  \frac{\partial \eta}{\partial \nu}=\frac{\partial \xi}{\partial\nu}=0,&x\in\partial\Omega, t>0,
 \end{array}\right.
 \nonumber
 \end{equation}
with $\eta(x,t)=S(x,t)-\tilde S(x)$ and $\xi(x,t)=I(x,t)$. Now suppose that
 $(\eta(x,t),\xi(x,t))= (e^{-\lambda t}\phi(x),e^{-\lambda t}\psi(x) )$
is a solution of the above linear system with $\lambda$ being a complex number. Then simple calculations show that
 \begin{equation}
 \left\{ \begin{array}{lll}
 d_S\Delta\phi - \phi+ (-\beta\tilde S+\gamma )\psi+\lambda\phi=0,&x\in\Omega,\\
 d_I\Delta\psi+ (\beta\tilde S-\gamma-\mu )\psi+\lambda\psi=0,&x\in\Omega,\\
  \frac{\partial \phi}{\partial \nu}=\frac{\partial\psi}{\partial \nu}=0,&x\in\partial\Omega.
 \end{array}\right.
 \label{linear-ellip-DFE}
 \end{equation}

We first assume that $\mathcal R_0<1$ and shall show that $(\tilde S,0)$ is linearly stable;
that is, if $(\lambda,\phi,\psi)$ is any solution of \eqref{linear-ellip-DFE} with $\phi$ or $\psi$ not identically zero,
then ${\mbox Re}(\lambda)>0.$ There are two cases to consider: $\psi\equiv 0$ and $\phi\not\equiv 0$; $\psi\not\equiv 0$.

In the former case, clearly $(\lambda,\phi)$ is an eigenpair of the eigenvalue problem
 \begin{equation}
 d_S\Delta u -u+\lambda u=0,\ x\in\Omega; \ \
 \frac{\partial u}{\partial \nu}=0,\ x\in\partial\Omega.
 \label{eigen-prob1}
 \end{equation}
It is obvious that $\lambda$ must be real due to the self-adjoint property of the operator involved in \eqref{eigen-prob1} and hence $\lambda\geq 1$, as we wanted.
If the latter case happens, it follows that $(\lambda,\psi)$ is an eigenpair of the eigenvalue problem \eqref{eigen-prob}
and hence $\lambda$ is real and $\lambda\geq \lambda^*>0$ due to Proposition \ref{property} (d). Thus, the linear stability of
$(\tilde S,0)$ is proved.

We now suppose $\mathcal R_0>1$ and show the instability of $(\tilde S,0)$.
Proposition \ref{property}(d) yields that $\lambda^*<0$. It is well known that the following linear problem
 \begin{equation}
 d_S\Delta \phi-\phi+\lambda^*\phi=(\beta\tilde S-\gamma)\psi^*,\ x\in\Omega,\ \
 \frac{\partial \phi}{\partial \nu}=0,\ x\in\partial\Omega
 \nonumber
 \end{equation}
admits a solution $\phi^*$. Consequently, $(\lambda^*,\phi^*,\psi^*)$ becomes a solution of \eqref{linear-ellip-DFE}
with $\lambda^*<0$ and $\psi^*>0$ and so $(\tilde S,0)$ is linearly unstable.
\end{proof}

Based on the \lq\lq ultimately uniformly boundedness\rq\rq\, \eqref{global-bdd-eq2},
we are able to establish the uniform persistence property of \eqref{model}
when the basic reproduction number $\mathcal R_0>1$.
In fact, one can easily adapt the arguments of \cite[Theorem 3.3]{PZ-non-12}, developed by Magal and Zhao
(see \cite[Theorem 4.5]{Magal-Zhao} and \cite[Chapter 13]{Zhao}), to conclude the following assertion.

\begin{theorem}
\label{uni-persis}
Suppose that $\mathcal R_0>1$. Then system \eqref{model} is uniformly persistent, i.e.,
there exists some $\eta>0$ independent of the initial data $(S_0,I_0)$, such that
$$\liminf_{t\to\infty}S(x,t)\geq \eta\quad \mbox{and}\quad \liminf_{t\to\infty}I(x,t)\geq \eta\quad \mbox{uniformly for }x\in\overline\Omega.$$
Furthermore, \eqref{model} admits at least one EE provided that $\mathcal R_0>1$.
\end{theorem}

\subsection{Global stability in homogeneous environment}

In this subsection, we consider the global stability of the DFE and EE of \eqref{model} in homogeneous environment,
i.e., all of parameters $\Lambda$, $\beta,\gamma$ and $\mu$ are positive constants. In view of \eqref{r0},
we now have an explicit expression for the basic reproduction number
$\mathcal R_0=\frac{\Lambda \beta}{\gamma+\mu}$
and the unique DFE is given by $(\tilde S,0)=(\Lambda,0)$. On the other hand, there exists a unique constant EE
$(\hat S, \hat I)$ if and only if $\mathcal R_0>1$,
where
 $$
 \hat S = \frac{\gamma+\mu}{\beta}=\frac{\Lambda}{\mathcal R_0} \quad \mbox{and}\quad \hat I
 =\frac{\Lambda}{\mu}\left(1-\frac{1}{\mathcal R_0}\right)=\frac{\gamma+\mu}{\mu \beta} \left(\mathcal R_0-1\right).
 $$

For later purpose, we recall a simple fact which can be found in \cite[Lemma 2.5.1]{Wang}:

\begin{lemma}
\label{lem-Lya}
Let $a$ and $b$ be positive constants. Assume that $\varphi, \psi\in C^1([a,\infty))$, $\psi(t)\geq 0$ in $[a,\infty)$
and $\varphi$ is bounded from below. If $\varphi'(t)\leq -b \psi(t)$ and $\psi'(t)\leq K$ in $[a,\infty)$ for some constant $K$, then $\lim_{t\to\infty}\psi(t)=0.$
\end{lemma}

By constructing suitable Lyapunov functionals, we can show
\begin{theorem}
\label{glo-sta}
Assume that $d_S=d_I$. Then the following assertions hold.
\begin{enumerate}
\item[\rm {(i)}] If $\mathcal R_0\leq 1$, then the DFE is globally attractive;
\item[\rm {(ii)}] If $\mathcal R_0>1$, then the EE is globally attractive.
\end{enumerate}
\end{theorem}

\begin{proof} Set $d_S=d_I=d$. To verify (i), for any solution $(S,I)$ of \eqref{model}, we define
 $$
V(t)=\frac{1}{2} \int _{\Omega}\left[(S-\Lambda)+I \right]^2 dx + \frac{\mu+1}{\beta} \int_\Omega I dx.
 $$
Then, for all $t>0$, direct calculations show that
 \begin{align}\label{PW-hom-Ly1a}
 V'(t) &= \int _{\Omega}\left[(S-\Lambda)+I\right] (S_t+I_t) dx +\frac{\mu+1}{\beta} \int_\Omega I_t dx \nonumber \\
 &=  \int _{\Omega}\left[(S-\Lambda)+I \right] (d_S\Delta S+\Lambda - S -\mu I+d_I\Delta I)  dx
 +\frac{\mu+1}{\beta} \int_\Omega (d_I\Delta I + \beta SI-\gamma I - \mu I)dx \nonumber \\
 & =  -d \int _{\Omega} |\nabla (S+I)|^2 dx -\int_\Omega (S-\Lambda)^2 dx -\mu \int_\Omega I(S-\Lambda)d x +\int_\Omega I(\Lambda -S) dx\nonumber\\
& \quad \,\,  -\mu \int_\Omega I^2 dx +\frac{\mu+1}{\beta} \int_\Omega (\beta SI-\gamma I -\mu I)dx \nonumber \\
 & \leq -\int _{\Omega}(S-\Lambda)^2 dx-\mu \int_\Omega I^2 dx+\frac{\mu+1}{\beta}\left[\beta \Lambda-(\gamma+\mu)\right] \int_\Omega I dx\leq 0,
 \nonumber
 \end{align}
due to the assumption that $\mathcal R_0=\beta \Lambda/(\gamma+\mu)\leq 1$.
Define
$$\psi(t)=\int_\Omega (S-\Lambda)^2 dx+\mu \int_\Omega I^2 dx\geq 0.$$
Recall that Theorem \ref{global-bdd} tells us that both $\|S(\cdot,t)\|_{L^\infty(\Omega)}$ and $\|I(\cdot,t)\|_{L^\infty(\Omega)}$ are bounded.
Hence, by \cite[Theorem A2]{BDG}, we have
\begin{equation}
 \|S(\cdot,t)\|_{C^{2+\alpha}(\overline\Omega)}+\|I(\cdot,t)\|_{C^{2+\alpha}(\overline\Omega)}\leq C_0,\ \ \forall t\geq1,
 \label{SI-C2}
\end{equation}
for some positive constant $C_0$. Furthermore, using both PDEs of \eqref{model}, one can easily see that $\psi'(t)$ is bounded from above for $t\in [1,\infty)$.
We deduce from Lemma \ref{lem-Lya} (by taking $\varphi(t)=V(t)$) that
 $$
 \left(S(x,t),I(x,t)\right)\to(\Lambda,0)=(\tilde S,0)\ \ \mbox{in}\, \left(L^2(\Omega)\right)^2,\ \ \mbox{as}\ t\to\infty.
 $$
Furthermore, \eqref{SI-C2} indicates that $(S(\cdot,t),I(\cdot,t))$ is compact in $C^2(\overline\Omega)$
for $t\geq 1$. This, together with the above $L^2$-convergence, yields that
$$
 \left(S(x,t),I(x,t)\right)\to (\tilde S,0)\ \ \mbox{in}\, \left(C^2(\overline\Omega)\right)^2,\ \ \mbox{as}\ t\to\infty;
 $$
that is, $(\tilde S, 0)$ attracts all solutions of \eqref{model}.

We next prove (ii). Define
$$
W(t)=\frac{1}{2} \int_\Omega \left[\left(S-\hat S\right)+\left(I-\hat I\right) \right]^2 dx+\frac{\mu+1}{\beta} \int_\Omega \left(I-\hat I-\hat I\ln\frac{I}{\hat I} \right)dx\geq 0, \ \ \forall t>0.
$$
By straightforward computations, we have
\begin{align}
W'(t)&=\int_\Omega \left[\left(S-\hat S\right)+\left(I-\hat I\right) \right] (S_t+I_t)dx+
\frac{\mu+1}{\beta} \int_\Omega \left(1-\frac{\hat I}{I}\right) I_t dx\nonumber\\
 &= \int_\Omega \left[\left(S-\hat S\right)+\left(I-\hat I\right) \right] (d \Delta S +\Lambda -S -\mu I +d \Delta I)dx\nonumber\\
&\quad \,\, + \frac{\mu+1}{\beta} \int_\Omega \left(1-\frac{\hat I}{I}\right) (d\Delta I + \beta SI-\gamma I -\mu I)dx \nonumber\\
&=-d\int_\Omega |\nabla (S+I)|^2 dx -\frac{\mu+1}{\beta} d \hat I\int_\Omega\frac{|\nabla I|^2}{I^2} dx +\frac{\mu+1}{\beta}\int_\Omega (I-\hat I)(\beta S -\beta \hat S)dx \nonumber\\
&\quad\,\, +\int _\Omega \left[(S-\hat S)+(I-\hat I)\right] \left(\hat S+\mu \hat I-S -\mu I\right)dx \nonumber\\
& \leq -\int_\Omega (S-\hat S)^2 dx - \mu \int_\Omega (I-\hat I)^2 dx \leq 0, \nonumber
\end{align}
where we have used the fact that
$\Lambda =\hat S+\mu \hat I$ and $\gamma+\mu = \beta \hat S.$

In  Lemma \ref{lem-Lya}, let
$$\phi(t)=W(t),\ \ \ \psi(t)= \int_\Omega (S-\hat S)^2 dx + \mu \int_\Omega (I-\hat I)^2 dx, \ \ \forall t>0.$$
Then arguing similarly as before, we eventually conclude that
$$
\left(S(x,t),I(x,t)\right)\to (\hat S,\hat I)\ \ \mbox{in}\, \left(C^2(\overline\Omega)\right)^2,\ \ \mbox{as}\ t\to\infty.
$$
The proof is complete.
\end{proof}

The above theorem tells us that system \eqref{model} is uniformly persistent in
homogeneous environment provided $\mathcal R_0>1$, at least in the equal diffusion rate case.

\begin{remark}\em{
For general positive functions $\Lambda,\,\beta,\,\gamma,\,\mu$ and constants $d_S,\,d_I>0$, we suspect that
\eqref{model} has a unique EE which is globally attractive if $\mathcal R_0>1$, and the DFE is globally attractive if $\mathcal R_0\leq1$. However the justification of this suspicion is highly nontrivial and has to be left open in the current paper. }
\end{remark}

\section{Asymptotic profile of EE}
\label{sec-3}
 In this section, we are concerned with the asymptotic behavior of EE of \eqref{model},
which is a positive solution to the elliptic system:
 \begin{equation}\label{model-ss}
 \left\{ \begin{array}{lll}
  -d_S\Delta S=\Lambda(x)- S -\beta(x)SI+\gamma(x)I, & x\in \Omega,\\
  -d_I\Delta I=\beta(x)SI -[\gamma(x)+\mu(x)] I, & x\in \Omega,\\
  \frac{\partial S}{\partial\nu}=\frac{\partial I}{\partial \nu}=0,& x\in \partial\Omega
 \end{array}\right.
 \end{equation}
as one of the diffusion rates $d_S,\,d_I$ goes to zero.

\subsection{The case of $d_S\to 0$}

Using a singular perturbation argument, one can easily show that $\tilde S$, being the unique positive solution of \eqref{dfe},
converges uniformly to $\Lambda$ as $d_S\to 0$ (see \cite[Lemma 3.2]{PSW-non-08}).
Therefore, according to the continuity of eigenvalues with respect to the potential function,
we see that the principal eigenvalue $\lambda^*$ of \eqref{eigen-prob} converges to the principal eigenvalue of the following eigenvalue problem
 \begin{equation}
 d_I\Delta u+\left(\beta \Lambda-\gamma-\mu\right)u+\lambda u=0,\ x\in\Omega;\ \
 \frac{\partial u}{\partial\nu}=0,\ x\in\partial\Omega,
 \label{eigen-prob2}
 \end{equation}
which is denoted by $\lambda_0$. To ensure the existence of EE for all small $d_S$, one has to assume $\lambda_0<0$.

Now we are ready to establish the main result of this subsection.

\begin{theorem}\label{model-ss-ds}
Assume that $\lambda_0<0$.
Fix $d_I>0$, and let $d_S\rightarrow0$, then every positive solution
$(S,I)$ of \eqref{model-ss} satisfies {\rm(}up to a subsequence of $d_S\rightarrow0${\rm )}
 $$
 \left(S,I\right)\rightarrow \left(\underline S,\underline I\right)\ \ \mbox{uniformly on } \overline{\Omega},
 $$
where $\underline S(x)=\frac{\Lambda (x)+ \gamma \underline I(x)}{1+\beta \underline I(x)},$
and $\underline I$ is a positive solution to
 \begin{equation}
 \label{model-ss-ds-small}
 -d_I\Delta  \underline I=\beta(x) \underline S \underline I-(\gamma(x)+\mu(x)) \underline I,\ x\in \Omega;\ \
 \frac{\partial \underline I}{\partial \nu}=0,\ x\in \partial\Omega.
 \end{equation}
\end{theorem}

\begin{proof} Mentioned as before, \eqref{model-ss} has at least one EE for all small $d_S>0$ when $\lambda_0<0$.
In the following, we divide our argument into three steps for sake of clarity.

{\bf Step 1: A priori bounds for $S$ and $I$.}
Assume $S(x_0)=\max_{x\in\overline\Omega} S(x)$. We apply the maximum principle \cite[Proposition 2.2]{LN-JDE-96}
to the first equation of \eqref{model-ss} to derive
$
\Lambda(x_0)-S(x_0)-\beta (x_0) S(x_0) I(x_0)+\gamma (x_0) I(x_0)\geq 0,
$
or,
\begin{equation}
\Lambda^*\geq \Lambda(x_0) \geq S(x_0) + I(x_0) \left(\beta (x_0) S(x_0) -\gamma (x_0) \right).
\label{thm5-1}
\end{equation}
If $\beta (x_0) S(x_0) -\gamma (x_0)\leq 0$, then
 $ \max_{\overline\Omega}S = S(x_0)\leq \gamma(x_0)/\beta(x_0)\leq
\| \gamma/\beta\|_{L^\infty(\Omega)}.$
If $\beta (x_0) S(x_0) -\gamma (x_0)> 0$, it follows from \eqref{thm5-1} that
$\max_{\overline\Omega}S = S(x_0)\leq \Lambda^*.$ Thus, for any $d_S,\,d_I>0$, we have
\begin{equation}
\max_{\overline\Omega}S \leq \max\left\{ \Lambda^*, \  \ \left\| \frac{\gamma}{\beta} \right\|_{L^\infty(\Omega)} \right\}.
\label{thm5-2}
\end{equation}

On the other hand, set $S(x_1)=\min_{x\in \overline\Omega}S(x)$. Then an application of the maximum principle \cite[Proposition 2.2]{LN-JDE-96} implies that
$\Lambda(x_1)-S(x_1)-\beta (x_1) S(x_1) I(x_1)+\gamma (x_1) I(x_1)\leq 0$,
equivalently,
\begin{equation}
\frac{\Lambda(x_1)+\gamma (x_1) I(x_1)}{1+\beta(x_1) I(x_1)}\leq S(x_1).
\nonumber
\end{equation}
Obviously, there exists a positive constant $c_*$, independent of $d_S,\,d_I>0$, such that
\begin{equation}
c_*\leq\frac{\Lambda(x_1)+\gamma (x_1) I(x_1)}{1+\beta(x_1) I(x_1)}.
\nonumber
\end{equation}
Hence, for any $d_S,\,d_I>0$, it holds
\begin{equation}
c_*\leq S(x), \ \ \ \forall x\in\overline\Omega.
\label{thm5-2-a}
\end{equation}

Integrating both PDEs of \eqref{model-ss} over $\Omega$ yields
 $$
 \int_\Omega\left\{\Lambda(x)- S -\beta(x)SI+\gamma(x)I\right\}dx=0,\ \ \int_\Omega\left\{\beta(x)SI-[\gamma(x)+\mu(x)]I\right\}dx=0,
 $$
from which it immediately follows that
\begin{equation}
\mu_* \int_\Omega I dx\leq \int_\Omega \mu I dx +\int_\Omega S dx  = \int_\Omega \Lambda dx\leq |\Omega| \Lambda^*
\label{thm5-3}
\end{equation}
and
 \begin{equation}
\beta_* \int_\Omega SI dx\leq \int_\Omega\beta SI dx\leq (\gamma^*+\mu^*)\int_\Omega I dx\leq\frac{ |\Omega| \Lambda^*(\gamma^*+\mu^*)}{\mu_*}.
\label{thm5-3-a}
\end{equation}

We now write the $I$-equation as
\begin{equation}
\dis -\Delta I = \frac{1}{d_I} \left[\beta S -(\gamma+\mu) \right] I,\ x\in\Omega;\ \
\dis \frac{\partial I}{\partial \nu}=0,\ x\in\partial \Omega.
\label{thm5-4}
\end{equation}
According to the Harnack-type inequality (see, e.g., \cite{Lieberman-SIAM-05} or \cite[Lemma 2.2]{PSW-non-08}), \eqref{thm5-2} and \eqref{thm5-3}, we are led to
 \begin{equation}
 \max_{\overline{\Omega}}I\leq C\min_{\overline{\Omega}}I \leq C \frac{1}{|\Omega|}\int_\Omega I dx \leq C.
\label{thm5-5}
 \end{equation}
Hereafter, $C$ represents a positive constant independent of small $d_S>0$ which may vary from place to place.

{\bf Step 2: Convergence of $I$.} Recall that $I$ satisfies \eqref{thm5-4}.
By \eqref{thm5-2} and \eqref{thm5-5}, we have
 \begin{equation*}
 \left\|\frac{1}{d_I}\left[\beta S - (\gamma+\mu) \right]  I\right\|_{L^p(\Omega)} \leq C,\quad\forall\,p>1.
 \end{equation*}
From the standard $L^p$-estimate for elliptic equations (see, e.g., \cite{GT}), it follows that
$\left\|I\right\|_{W^{2,p}(\Omega)}\leq C$ for any given $p>1$.
Taking $p$ to be sufficiently large, we see from the Sobolev embedding that
$\left\|I\right\|_{C^{1+\alpha}(\overline\Omega)}\leq C$ for some $0<\alpha<1$.
As a result, there exists a subsequence of $d_S\to 0$, say $d_n:=d_{S,n}$, satisfying $d_n\to 0$ as $n\rightarrow\infty$, and
a corresponding positive solution $(S_n,I_n)$ of \eqref{model-ss} with $d_S=d_n$, such that
 \begin{equation}
 \label{thm5-7}
 I_n\rightarrow \underline I\ \ \mbox{uniformly on}\ \overline{\Omega},\ \ \mbox{as}\ n\to\infty,
 \end{equation}
where $0\leq \underline I\in C^1(\overline{\Omega})$. In view of  \eqref{thm5-5},
 \begin{equation}
 \label{thm5-8}
 \mbox{either}\ \underline I\equiv0\ \ \mbox{on}\ \overline{\Omega}\ \ \mbox{or}\ \ \underline I>0\ \ \mbox{on}\ \overline{\Omega}.
 \end{equation}

Suppose the former holds in \eqref{thm5-8}; that is,
 \begin{equation}
 \label{thm5-9}
 I_n \rightarrow 0\ \ \mbox{uniformly on}\ \overline{\Omega},\ \ \mbox{as}\ n\rightarrow \infty.
 \end{equation}
Then for sufficiently small $\epsilon>0$, we have
 $ 0\leq I_n(x)\leq\epsilon,\, \forall x\in\overline\Omega$, for all large $n$.
This fact, together with the first equation of \eqref{model-ss}, implies that for all large $n$, $S_n$ satisfies
 \begin{equation*}
 \dis -d_n\Delta S_n\leq \Lambda - S_n+\gamma^*\epsilon,\ x\in \Omega;\ \ \
 \dis \frac{\partial S_n}{\partial \nu}=0,\  x\in \partial\Omega
 \end{equation*}
and
 \begin{equation*}
 \dis -d_n\Delta S_n\geq \Lambda - S_n- \beta^*\epsilon S_n,\ x\in \Omega;\ \ \
 \dis \frac{\partial S_n}{\partial \nu}=0,\ x\in \partial\Omega.
 \end{equation*}
We consider the following two auxiliary problems:
 \begin{equation}
 \dis -d_n\Delta u = \Lambda - u+\gamma^*\epsilon,\ x\in \Omega;\ \ \
 \dis \frac{\partial u}{\partial \nu}=0,\  x\in \partial\Omega,
\label{thm5-10}
 \end{equation}
and
 \begin{equation}
 \dis -d_n\Delta v = \Lambda - v - \beta^*\epsilon v,\ x\in \Omega;\ \ \
 \dis \frac{\partial v}{\partial \nu}=0,\ x\in \partial\Omega.
\label{thm5-11}
 \end{equation}
It is clear that systems \eqref{thm5-10} and \eqref{thm5-11} admit a unique positive solution, denoted by $u_n$ and $v_n$, respectively. A simple sub-supsolution argument, combined with the uniqueness, guarantees that
$v_n\leq S_n \leq u_n$ on $\overline\Omega$ for all large $n$.
Using a singular perturbation argument as in \cite[Lemma 2.4]{DPW-JDE}, it can be shown that
 $$
 u_n\to \Lambda + \gamma ^* \epsilon,\ \ v_n\to\frac{ \Lambda}{1+\beta^*\epsilon}\quad \mbox{uniformly on }\overline\Omega,\mbox{ as }n\to \infty.
 $$
Sending $n\to\infty$, we find
$$\frac{\Lambda(x)}{1+\beta^*\epsilon}\leq \liminf_{n\to \infty}S_n(x)\leq \limsup_{n\to\infty}S_n(x) \leq \Lambda(x)+\gamma^*\epsilon.$$
Thanks to the arbitrariness of small $\epsilon>0$, we obtain that
\begin{equation}
S_n \to \Lambda \ \ \mbox{uniformly on }\overline\Omega, \ \mbox{as } n\to \infty.
\label{thm5-13}
\end{equation}

Observe that $I_n$ fulfills
 \begin{equation}
 -d_I\Delta I_n=\beta(x)S_n I_n -(\gamma+\mu) I_n,\ x\in\Omega; \quad \frac{\partial I_n}{\partial \nu}=0,\ x\in\partial\Omega.
\label{thm5-135}
 \end{equation}
Define $\tilde{I}_n:=\frac{I_n}{\|I_n\|_{L^\infty(\Omega)}}$.
Then $\|\tilde{I}_n\|_{L^\infty(\Omega)}=1$
for all $n\geq 1$, and $\tilde{I}_n$ solves
 \begin{equation}
 \label{thm5-14}
 \dis -d_I\Delta\tilde{I}_n=\left[\beta(x)S_n - (\gamma+\mu) \right]\tilde{I}_n, \ x\in \Omega;\ \ \
 \dis \frac{\partial \tilde{I}_n}{\partial \nu}=0,\ x\in \partial\Omega.
 \end{equation}
As before, through a standard compactness argument for elliptic equations, after passing to a further subsequence if necessary, we may assume that
 \begin{equation*}
 \label{PW-elliptic-estimate-12k}
 \tilde{I}_n\rightarrow \tilde{I}\ \ \mbox{in}\ C^1(\overline{\Omega}),\ \ \mbox{as}\ n\rightarrow \infty,
 \end{equation*}
where $0\leq \tilde{I}\in C^1(\overline{\Omega})$ with $\|\tilde{I}\|_{L^\infty(\Omega)}=1$. By \eqref{thm5-13} and
\eqref{thm5-14}, $\tilde{I}$ satisfies
 \begin{equation}
 \label{thm5-15}
 \dis -d_I\Delta\tilde{I}=\left[\beta \Lambda-(\gamma+\mu) \right]\tilde{I}, \ x\in \Omega;\ \ \
 \dis \frac{\partial\tilde{I}}{\partial \nu}=0,\ x\in \partial\Omega.
 \end{equation}
The Harnack-type inequality (see, \cite{Lieberman-SIAM-05} or \cite[Lemma 2.2]{PSW-non-08}) applied to \eqref{thm5-15} yields $\tilde{I}>0$ on $\overline{\Omega}$.
However, the positiveness of $\tilde I$ indicates that the principal eigenvalue $\lambda_0$ of the eigenvalue problem \eqref{eigen-prob2} must be zero (with $\tilde I$ being a corresponding eigenfunction), contradicting our assumption that
$\lambda_0<0$. Thus, \eqref{thm5-9} cannot occur,
and we must have $\underline I>0\ \mbox{on}\ \overline{\Omega}$. That is,
 \begin{equation}
 \label{thm5-16}
 I_n\rightarrow \underline I>0\ \ \mbox{uniformly on}\ \overline{\Omega},\ \mbox{as}\ n\rightarrow \infty.
 \end{equation}

{\bf Step 3: Convergence of $S$.} Notice that $S_n$ solves
\begin{equation}
  -d_n\Delta S_n=\Lambda - S_n-\beta S_n I_n +\gamma I_n,\ x\in\Omega;\ \
  \frac{\partial S_n}{\partial\nu}=0,\ x\in\partial\Omega.
 \label{thm5-17}
 \end{equation}
In view of \eqref{thm5-16}, we see that for any small $\epsilon>0$, it holds
 \begin{equation}
 0<\underline I(x)-\epsilon\leq I_n(x)\leq \underline I(x)+\epsilon,\ \ \forall x\in\overline\Omega
 \label{thm5-18}
 \end{equation}
for all large $n$.
Thus, for all sufficiently large $n$, we have
$$\Lambda - S_n -\beta S_n (\underline I+\epsilon)+\gamma (\underline I-\epsilon)\leq \Lambda - S_n - \beta S_n I_n+\gamma I_n \leq \Lambda -S_n -\beta S_n (\underline I-\epsilon)+\gamma (\underline I+\epsilon).$$

Given large $n$, we consider the following auxiliary problem
 \begin{equation}
-d_n\Delta w=\Lambda - w -\beta w (\underline I+\epsilon)+\gamma (\underline I-\epsilon) , \ x\in\Omega;\quad \frac{\partial w}{\partial \nu}=0,\ x\in\partial\Omega.
\label{thm5-19}
 \end{equation}
It is clear that \eqref{thm5-19} admits a unique positive solution, denoted by $\overline w_n$.
By similar arguments to those in the proof
of \cite[Lemma 2.4]{DPW-JDE}),
we notice that
 \begin{equation*}
 \overline w_n\rightarrow \frac{\Lambda + \gamma (\underline I-\epsilon)}{1+\beta (\underline I+\epsilon)}   \ \ \mbox{uniformly on}\ \overline{\Omega},\ \mbox{as}\ n\rightarrow \infty.
 \end{equation*}
Since $S_n$ is an upper solution of \eqref{thm5-19}, it then follows that
 \begin{equation}
 \label{thm5-20}
 \liminf_{n\to\infty}S_n(x) \geq \lim_{n\to \infty}\overline w_n(x)=\frac{\Lambda(x) + \gamma(x) (\underline I(x)-\epsilon)}{1+\beta(x) (\underline I(x)+\epsilon)}  \ \ \ \mbox{uniformly on}\ \overline{\Omega}.
 \end{equation}
Similarly, one can further show that
\begin{equation}
 \label{thm5-21}
 \limsup_{n\to\infty}S_n(x) \leq \frac{\Lambda(x) + \gamma(x) (\underline I(x)+\epsilon)}{1+\beta(x) (\underline I(x)-\epsilon)}  \ \ \ \mbox{uniformly on}\ \overline{\Omega}.
 \end{equation}
In view of \eqref{thm5-20} and \eqref{thm5-21}, combined with the arbitrariness of small $\epsilon>0$, we have
$$\lim_{n\to\infty}S_n(x) = \underline S(x):= \frac{\Lambda (x)+\gamma (x) \underline I(x)}{1+\beta(x)\underline I(x)}\ \ \ \mbox{uniformly on}\ \overline{\Omega}.$$
Because of \eqref{thm5-135}, it can be easily seen that $\underline I$ satisfies \eqref{model-ss-ds-small}. The proof is complete.
\end{proof}

\subsection{The case of $d_I\to 0$}

This subsection is devoted to the investigation of the asymptotic behavior of positive solutions of \eqref{model-ss} with $d_S>0$ being fixed and $d_I\to 0$.
Because of mathematical difficulty, we can only deal with one space dimension case, that is, the habitat $\Omega$ is an interval.
Without loss of generality, we take $\Omega=(0,1)$.

In light of \ref{property} (a) and \ref{uni-persis}, we assume that
$\{\beta(x)\tilde S(x)>{\gamma(x)+\mu(x)}:\ x\in[0,1]\}$ is non-empty so that $\mathcal{R}_0>1$ and thus \eqref{model-ss} admits positive solutions for all small $d_I>0$.
Our main result reads as follows.

\begin{theorem}
\label{model-ss-di}
Assume that the set $\{x\in [0,1]: \ \beta(x)\tilde S(x)>\gamma(x)+\mu(x)\}$ is non-empty.
Fix $d_S>0$ and let $d_I\rightarrow0$, then every positive solution $\left(S,I\right)$
of \eqref{model-ss} satisfies (up to a subsequence of $d_I$) that
$S\to S_0\ \ \mbox{uniformly on}\ [0,1],$
where $S_0\in C([0,1])$ and $S_0>0$ on $[0,1]$, and
$\int_0^1 Idx\to I_0$ for some positive constant $I_0$.
\end{theorem}

\begin{proof}
Notice that \eqref{thm5-2}, \eqref{thm5-2-a}, \eqref{thm5-3} and \eqref{thm5-3-a} remain true in the current situation.
Since the spatial domain is one dimensional and $S$ satisfies
\begin{equation}
\dis -d_SS''(x) +S(x)= \Lambda -\beta  S(x)I(x)+\gamma I(x),\ x\in(0,1);\  \
\dis S'(0)=S'(1)=0,
\label{thm6-1}
\end{equation}
we deduce from the elliptic $L^1$-theory in \cite{BW} that, for any $p>1,$
$\|S \|_{W^{1,p}(0,1)}\leq C$, where $C$ is a positive constant independent of $d_I$ but allows to be different below.
Then for sufficiently large $p$,
the Sobolev embedding theorem guarantees that $\|S\|_{C^{\alpha}([0,1])}\leq C$ for some $\alpha\in(0,1)$.
Moreover, up to a sequence of $d_I\to0$,
say $d_n:=d_{I,n}\to 0$ with $d_n\to 0$ as $n\to\infty$,
the corresponding positive solution sequence $(S_n,I_n)$ of \eqref{model-ss} with $d_I=d_{n}$ satisfies $ S_n\rightarrow S_0>0$ in $C([0,1])$, as $n\to\infty$
due to \eqref{thm5-2-a}.

In light of \eqref{thm5-3}, by passing a subsequence of $d_n$ if necessary, we may assume that
 $\int_0^1 I_ndx\to I_0$, as $n\to\infty,$
for some nonnegative constant $I_0$.
To show $I_0>0$, we proceed indirectly and suppose that $I_0=0$.
By integrating \eqref{thm6-1} from $0$ to $x$, we have
 $$
 S'_n(x)=-\frac{1}{d_S}\int_0^x\{\Lambda(y)-S_n(y)-\beta(y)S_n(y)I_n(y)+\gamma(y)I_n(y)\}dy,\ \ \forall x\in[0,1].
 $$
By sending $n\to\infty$ and using $\int_0^1 I_ndx\to0$, it then follows
 $$
 S'_n(x)\to -\frac{1}{d_S}\int_0^x [\Lambda(y)-S_0(y)] dy\  \ \mbox{uniformly on}\ [0,1].
 $$
As $S_n(x)-S_n(0)=\int_0^xS_n'(y)dy$ for any $n\geq1$, we find that $S_0$ solves
 $$
 S_0(x)-S_0(0)=-\frac{1}{d_S}\int_0^x\left\{\int_0^y [\Lambda(z)-S_0(z)] dz\right\}dy,
 $$
which in turn implies that
 \begin{equation}
\dis -d_SS_0''(x)= \Lambda(x)-S_0(x),\ x\in(0,1);\  \
\dis S'_0(0)=0.
\label{thm6-1-a}
\end{equation}
If integrating \eqref{thm6-1} from $x$ to $1$, one can use the analysis similar as above to know that
$S_0'(1)=0$. Therefore, this and \eqref{thm6-1-a} give that $S_0=\tilde S$, that is, $S_n\to \tilde S$ uniformly on $[0,1]$ as $n\to\infty$.

On the other hand, observe that
 $$
 \lambda_1(d_n,\gamma(x)+\mu(x)-\beta(x) S_n(x))=0,\ \ \forall n\geq1,
 $$
where $\lambda_1(d_n,\gamma(x)+\mu(x)-\beta(x) S_n(x))$ stands for the principal eigenvalue of the following eigenvalue problem:
 \begin{equation}
 \displaystyle d_n\Delta u+\left[\beta(x) S_n(x)-\gamma(x)-\mu(x)\right]u+\lambda u=0,\ x\in(0,1);\ \
 \displaystyle u'(0)=u'(1)=0.
 \nonumber
 \end{equation}
Combined with the fact that the principal eigenvalue continuously depends on the parameters,
the argument as in \cite[Lemma 2.3]{Allen} yields
 $$
 0=\lambda_1(d_n,\gamma(x)+\mu(x)-\beta(x) S_n(x)) \to\min_{x\in[0,1]}\{\gamma(x)+\mu(x)-\beta(x)\tilde S(x)\},\ \ \mbox{as} \ n\to\infty,
 $$
contradicting our assumption $\min_{x\in[0,1]}\{\gamma(x)+\mu(x)-\beta(x)\tilde S(x)\}<0$. Thus, it is necessary that $I_0>0$.
The proof is complete.
\end{proof}

\section{Summary and Discussion}\label{dis-sec}
\subsection{Summary of analytical results}
In this paper, we are concerned with the SIS epidemic model \eqref{model} with mass action infection mechanism and linear source. To study the parabolic problem \eqref{model}, our first step is to establish the global existence and uniform boundedness of solutions. Then a basic reproduction number $\mathcal R_0$ is defined via a variational characterization, which determines the local stability of the unique DFE. When the environment is spatially homogeneous and the diffusion rates of the susceptible and infected are equal, by constructing suitable Lyapunov functionals, we further prove the global attractivity of the DFE for $\mathcal R_0 \leq 1$ and that of the EE for $\mathcal R_0>1$. We are mainly interested in the asymptotic behavior of positive steady states $(S,I)$ of problem \eqref{model}, which exist provided $\mathcal R_0>1$ in general heterogeneous environment, as the diffusion rates of the susceptible or the infected tends to zero. For fixed $d_I>0$, Theorem \ref{model-ss-ds} shows that the limiting functions of both $S$ and $I$ as $d_S\to 0$, are positive throughout the habitat. In the one dimensional interval, say $[0,1]$, for fixed $d_S>0$, Theorem \ref{model-ss-di} indicates that the limiting function of $S$ as $d_I\to 0$ is positive in $[0,1]$ while the total infected population tends to a positive constant.

Since there are four principle models \eqref{intro2}, \eqref{model}, \eqref{intro3} and \eqref{intro4} to model the SIS epidemic dynamics based on different infection mechanisms and modeling ideas, it will be helpful to summarize their results and make a comparison so as to understand the influence of the factors such as infection mechanism, movement rate and source term on the eradication of epidemics. Numerical simulations will be performed to validate theoretical results and to predict possible outcomes for those cases that remain unknown analytically. Then we discuss the implication of these theoretical and numerical findings from the disease control viewpoint. Since the results of the model \eqref{model} have been summarized above, below we shall briefly recall the results for the SIS models \eqref{intro2}, \eqref{intro3} and \eqref{intro4} obtained in literature.

\subsubsection{Results on \eqref{intro3}}
The steady state problem corresponding to \eqref{intro3} satisfies
 \begin{equation}
 \left\{ \begin{array}{llll}
  -d_S\Delta S=-\beta(x)\frac{SI}{S+I}+\gamma(x)I,&x\in\Omega,\\
  -d_I\Delta I=\beta(x)\frac{SI}{S+I}-\gamma(x)I,&x\in\Omega,\\
  \frac{\partial S}{\partial \nu}=\frac{\partial I}{\partial \nu}=0,&x\in\partial\Omega,\\
  \int_\Omega [S(x)+I(x)]dx=N.
 \end{array}\right.
 \label{AllenSIS-SS}
 \end{equation}
Hereafter, $N$ is a fixed positive constant, representing the total number of the susceptible and infected populations. That is $N=\int_\Omega(S(x)+I(x))dx$ is a constant.

As in \cite{Allen, Peng-Yi}, we introduce the notion of low/high/moderate risk site/domain. We say that $x$ is a {\it low (or high or moderate) risk site} if the local disease transmission rate $\beta(x)$ is lower (or higher or equal to) than the local disease recovery rate $\gamma(x)$. Let
$$H^- = \{x\in\Omega: \beta(x)<\gamma(x)\}\quad \mbox{and}\quad H^+=\{x\in\Omega: \beta(x)>\gamma(x)\}$$
denote the set of low-risk sites and high-risk sites, respectively.

Assume that both $H^-$ and $H^+$ are nonempty. The authors in \cite{Allen} defined the basic reproduction number
\begin{equation*}
\hat{\mathcal R}_0 = \sup_{0\neq \varphi \in H^1(\Omega)} \frac{\int_\Omega \beta \varphi^2dx}{\int_\Omega (d_I |\nabla\varphi|^2 + \gamma \varphi^2)dx}
\label{r0-old}
\end{equation*}
and showed that the unique DFE $(N/|\Omega|,0)$ is globally stable if $\hat{\mathcal R}_0<1$, while it is unstable and  a unique EE exists if $\hat{\mathcal R}_0>1$. Indeed, following the argument similar to \cite{Cui-Lam-Lou}, one can show that the uniform persistence property holds once $\hat{\mathcal R}_0>1$.

The asymptotic profile of the EE was also investigated in \cite{Allen} when the diffusivity of the susceptible individuals tends to zero. In particular, the result of \cite{Allen} shows that
\begin{itemize}
\item As $d_S\to0$, the unique positive solution $(S,I)$ (which exists if $\hat{\mathcal R}_0>1$) of \eqref{AllenSIS-SS} fulfills
 $(S,I) \to (\hat S,0)$ uniformly on $\overline\Omega$, where $\hat S$ satisfies a free boundary problem, is positive at all low-risk sites and is also positive at some (but not all) high-risk sites.
\end{itemize}
This result indicates that it may be possible to entirely eliminate the infectious disease by restricting the motility rate of the susceptible to be small.

Further asymptotics of the EE in other cases were obtained by Peng \cite{Peng} wherein it was shown that if $d_I\to 0$ and $d:=d_I /d_S \to d_0 \in [0,\infty]$, then the unique positive solution $(S,I)$ of \eqref{AllenSIS-SS} satisfies the following:
\begin{itemize}
\item If $d_0 = 0$, then
$$S \to \frac{N}{\int_\Omega \left[1+(\beta - \gamma)_+ \gamma^{-1} \right]}\quad \mbox{and}\quad I \to \frac{N (\beta - \gamma)_+ \gamma^{-1}}{\int_\Omega \left[1+(\beta - \gamma)_+ \gamma^{-1} \right]}$$
uniformly on $\overline\Omega$. In what follows, $(s)_+ = \max\{s,0\}$.
\item If $d_0\in (0,\infty)$, then
$$S \to \frac{N d_0 \left[1-A(d_0;x)\right]} {\int_\Omega \left[A(d_0;x)+d_0 (1-A(d_0;x)) \right]},\ \ I \to \frac{N A(d_0;x) } {\int_\Omega \left[A(d_0;x)+d_0 (1-A(d_0;x)) \right]}$$
uniformly on $\overline\Omega$,
where $A(d_0;x)=\frac{d_0 (\beta - \gamma)_+}{d_0(\beta - \gamma) + \gamma}$.

\item If $d_0 = \infty$, then $I\to 0$ uniformly on $\overline\Omega$, and
$S\to \frac{N[1- A(\infty;x)]}{\int_\Omega [1-A(\infty;x)]}$
uniformly on any compact subset of $H^-$ and $H^+$ respectively, where
$A(\infty;x)=\left\{ \begin{array}{ll}
0,&\mbox{if }x\in\overline{H^-};\\
1,&\mbox{if } x\in H^+.
\end{array}\right.$
\end{itemize}
Clearly the limiting function of $I$ when $d_I\to 0$ and $d\to d_0 \in [0,\infty)$ is positive on $H^+$ while zero on $\overline{H^-}$. In particular, if $d_I\to 0$ and $d_S>0$ is fixed, we are in the first scenario above. Thus, for model \eqref{intro3}, we may conclude that the optimal strategy of eliminating the infectious disease is to restrict the motility rate of the susceptible population, while restricting the motility of infected population can only eradicate the disease in low-risk and moderate-risk sites. Of course, another strategy is to set $d_I\to 0$ and $d_S\to 0$ while the susceptible moves relatively slower than the infected.


\subsubsection{Result on \eqref{intro4}} Now we consider the scenario that the susceptible individuals are allowed to
have birth and death, and look at  the SIS reaction-diffusion system \eqref{intro4} with
a linear external source. One of the main results in \cite{LiPengWang} states that \eqref{intro4} admits at least one EE $(S,I)$ if $\hat{\mathcal R}_0>1$, which is in fact a positive steady state of \eqref{intro4} satisfying
 \begin{equation}
 \left\{ \begin{array}{llll}
   -d_S\Delta S=\Lambda(x)-S-\beta(x)\frac{SI}{S+I}+\gamma(x)I,&x\in\Omega,\\
   -d_I\Delta I=\beta(x)\frac{SI}{S+I}-\gamma(x)I,&x\in\Omega,\\
   \frac{\partial S}{\partial \nu}=\frac{\partial I}{\partial \nu}=0,&x\in\partial\Omega.
 \end{array}\right.
 \label{LPW-SS}
 \end{equation}
Moreover, it was proved in \cite{LiPengWang} that
\begin{itemize}
\item As $d_S\to 0$, both limiting functions of $S$ and $I$ are inhomogeneous and positive on the entire habitat $\overline\Omega$;

\item As $d_I\to 0$, the limiting function of $S$ is positive on the entire habitat $\overline\Omega$ and that of $I$ is positive only on high-risk sites.
\end{itemize}

\subsubsection{Result on \eqref{intro2}} In \cite{Deng-Wu, Wu-Zou, LiBo1}, the authors treated the SIS system \eqref{intro2} with mass action and its steady state problem:
 \begin{equation}
 \left\{ \begin{array}{llll}
   -d_S\Delta S=-\beta(x)SI+\gamma(x)I,&x\in\Omega,\\
   -d_I\Delta I=\beta(x)SI-\gamma(x)I,&x\in\Omega,\\
   \frac{\partial S}{\partial \nu}=\frac{\partial I}{\partial \nu}=0,&x\in\partial\Omega,\\
  \int_\Omega [S(x)+I(x)]dx=N.
 \end{array}\right.
 \label{Deng-Wu-SS}
 \end{equation}
For the mass action system \eqref{intro2}, the basic reproduction number depends on the total population size $N$ and is defined as
$$\tilde{\mathcal R}_0 = \sup_{0\neq \varphi\in H^1(\Omega)}\frac{(N/|\Omega|)\int_\Omega \beta \varphi^2}{\int_\Omega (d_I|\nabla \varphi|^2 +\gamma \varphi^2)} = \frac{N}{|\Omega|} \hat{\mathcal R}_0.$$
It is shown that a positive solution $(S,I)$ of \eqref{Deng-Wu-SS} exists whenever $\tilde{\mathcal R}_0>1$. Indeed, following the argument similar to \cite{Cui-Lam-Lou}, one can show the uniform persistence property holds once $\tilde{\mathcal R}_0>1$. Moreover, one can show that
$\tilde{\mathcal R}_0>1$ when $N> \int_\Omega \frac{\gamma(x)}{\beta(x)}dx$, and $\tilde{\mathcal R}_0>1$ is also possible when $N\leq\int_\Omega \frac{\gamma(x)}{\beta(x)}dx$ depending on the parameters $\beta,\,\gamma$ and $d_I$.
Furthermore, for fixed $d_I>0$, the following asymptotics as $d_S\to 0$ have been shown in \cite{LiBo1,Wu-Zou}:

\begin{itemize}
\item If either $N-\int_\Omega \frac{\gamma}{\beta}>\frac{1}{4}\int_\Omega \frac{|\nabla \beta|^2}{\beta^3}$ or $\frac{N}{|\Omega|}>\frac{\gamma}{\beta}$ on $\overline\Omega$, then
     $$
     (S,I) \to\left(\frac{\gamma(x)}{\beta(x)},\frac{N}{|\Omega|} -\frac{1}{|\Omega|} \int_\Omega \frac{\gamma(x)}{\beta(x)}dx \right)\ \ \ \mbox{uniformly on}\ \overline\Omega;
     $$

\item If $N\leq\int_\Omega \frac{\gamma(x)}{\beta(x)}dx$, then $(S,I) \to(S_*,0)$ uniformly on $\overline\Omega$, where $S_*$ is a positive function.

\end{itemize}
Under the assumption that $\Omega^+= \big\{ x\in\Omega:\frac{N}{|\Omega|}\beta(x)-\gamma(x)>0\big\}$ is nonempty, Wu and Zou \cite{Wu-Zou} further proved the following:

\begin{itemize}
\item If $d_I\to 0$ and $d_I/d_S\to d\in(0,\infty)$, then $(S,I) \to (S_{**},I_{**})$ uniformly on $\overline\Omega$ and $I_{**}$ is the unique nonnegative solution of
 $$
 \left\{ \frac{N}{|\Omega|}\beta - \gamma -\frac{(1-d)\beta}{|\Omega|}\int_\Omega I_{**} \right\}_+ - d \beta I_{**}=0,
 $$
and
 $$
 S_{**}= \frac{N}{|\Omega|} -\frac{(1-d)}{|\Omega|}\int_\Omega I_{**} -d I_{**}.
 $$

\end{itemize}
Therefore, the distribution of $I_{**}$ depends critically on the magnitude of $d$. In fact, if $d\in (0,1)$, then  $\{x\in\Omega:I_{**}(x)>0\}$ is a proper subset of $\Omega^+$; if $d\in(1,\infty)$, then $\Omega^+$ is a subset of $\{x\in\Omega:I_{**}(x)>0\}$; if $d=1$, then
 $$
 S_{**} = \frac{N}{|\Omega|} - \left(\frac{N}{|\Omega|}-\frac{\gamma}{\beta}\right)_+ \quad \mbox{and}\quad I_{**} = \left(\frac{N}{|\Omega|}-\frac{\gamma}{\beta}\right)_+.
 $$

On the other hand, in the case of one-dimensional domain, say $\Omega=(0,1)$, if $\gamma < N \beta$ on $[0,1]$, then for fixed $d_S>0$,
as $d_I\to 0$, the authors of \cite{LiBo1} proved that
\begin{itemize}
\item
Any EE $(S,I)$ satisfies $S\to \hat S$ uniformly on $[0,1]$ with a positive function $\hat S$ and $\int_0^1Idx$ converges to a positive constant.
\end{itemize}
Biologically, this implies that the infectious disease still persists when the movement of the infected population is small.

\subsection{Discussion and conclusions}
\subsubsection{Comparison of the basic reproduction number}
For readability, hereafter we call models \eqref{intro2}, \eqref{model}, \eqref{intro3}  and \eqref{intro4} and their corresponding EE problem (when no confusion is caused) as MO, MW, SO and SW, respectively, in order that each label of models can bear a meaning (see Table 1). For convenience, we also list the basic reproduction number for each of models MO, MW, SO, SW in Table \ref{t1}, where three observations are worth mentioning as follows.

(a) MO is the only one whose basic reproduction number depends on $N$ via $N/|\Omega|$ which measures the number of population per unit space. This implies that the total population plays a role in the eradication of diseases only for MO, and  also explains why a disease is easier to become endemic in a more crowded population than a sparse population as mentioned in \cite{Wu-Zou}. (b) If the birth-death effect is considered, then MO becomes MW whose basic reproduction number no longer depends on total population $N$. This indicates that the birth and death effects could be an important factor for the eradication of diseases in SIS models with mass-action infection mechanism. However, the birth-death effect is not important for models SO and SW any more, since both have the same basic reproduction number. (c) MW is the only model whose basic reproduction number depends (implicitly) on the diffusivity $d_S$ of the susceptibles.
\begin{table}[htbp]\label{t0}
\begin{center}
\begin{tabular}{|l||l||l|}
\hline
Model  &Infection Mechanisms &Basic Reproduction Number\\
\hline
MO=\eqref{intro2} &Mass-action incidence without birth-death &$\tilde{\mathcal R}_0 = \frac{N}{|\Omega|} \sup_{0\neq \varphi\in H^1(\Omega)}\frac{\int_\Omega \beta \varphi^2}{\int_\Omega (d_I|\nabla \varphi|^2 +\gamma \varphi^2)}$\\
\hline
MW=\eqref{model} &Mass-action incidence with birth-death &${\mathcal R}_0 = \sup_{0\neq \varphi\in H^1(\Omega)}\frac{\int_\Omega \beta \tilde{S}\varphi^2}{\int_\Omega (d_I|\nabla \varphi|^2 +(\gamma+\mu) \varphi^2)}$\\
\hline
SO\,=\,\eqref{intro3} &Standard incidence without birth-death  &$\hat{\mathcal R}_0 = \sup_{0\neq \varphi\in H^1(\Omega)}\frac{\int_\Omega \beta \varphi^2}{\int_\Omega (d_I|\nabla \varphi|^2 +\gamma \varphi^2)}$\\
\hline
SW=\eqref{intro4}  &Standard incidence with birth-death &$\hat{\mathcal R}_0 = \sup_{0\neq \varphi\in H^1(\Omega)}\frac{\int_\Omega \beta \varphi^2}{\int_\Omega (d_I|\nabla \varphi|^2 +\gamma \varphi^2)}$ \\
\hline
\end{tabular}
\end{center}
\caption{\small Basic reproduction numbers for SIS epidemic models, where $\tilde{S}$ in the basic reproduction number for MW is the unique solution of \eqref{dfe}.}
\label{t1}
\end{table}
\subsubsection{Asymptotic behavior of EE}
From the disease control point of view, one is mainly concerned with whether the infectious disease can be eradicated (namely whether $I(x)$ can go extinction either throughout the entire domain $\overline\Omega$ or partially). One of the strategies as recalled above is to control the motility of susceptible and/or infected populations. Below in Table \ref{t2} and Table \ref{t3} we capsulize the asymptotic behavior of EE $(S(x), I(x))$ as $d_S\to0$ or $d_I \to 0$ or both.  Furthermore we use numerical simulations to illustrate known results and predict possible outcomes for unknown cases. In the following, we shall use $(S^*,I^*)$ to represent the asymptotic behavior of EE for all models for simplicity.
We remark that the parameter values chosen in all simulations are sufficient to guarantee the existence of EE in models under consideration. For example, in Fig.\ref{fig1}, for any $d_S>0$ and $d_I>0$, $\hat {\mathcal R}_0=\tilde {\mathcal R}_0>\int_0^1 \beta(x)dx / \int_0^1 \gamma(x)dx = 1.5 / 1.2 >1$ and $\mathcal R_0> \int_0^1 \beta(x)\tilde S(x)dx / \int_0^1 [\gamma(x)+\mu(x)]dx = 4.5/2.2>1$.
\begin{table}[htb]
\begin{center}
\begin{tabular}{|l||p{6cm}||p{6cm}| }
\hline
Model  & Limit of $(S(x), I(x))$ as $d_S \to 0$ & Limit of $(S(x), I(x))$ as $d_I \to 0$\\[1.5mm]
\hline
MO & $S^*(x)>0$ and $I^*(x) \equiv 0$ (or $>0$) for small (or large) $N$& \multirow{2}{*}{$S^*(x)>0$ and $\int_\Omega I^*(x)>0$}\\[1.5mm]
\hline
MW & $S^*(x)>0$ and $I^*(x)>0$ &$S^*(x)>0$ and $\int_\Omega I^*(x)>0$ \\[1.5mm]
\hline
SO & $S^*(x)\geq 0$ and $I^*(x) \equiv 0$ & $S^*(x) > 0$ and $I^*(x)\equiv0$ iff $x \in \overline{H^-}$\\[1.5mm]
\hline
SW   & $S^*(x)>0$ and $I^*(x)>0$& $S^*(x)>0$ and $I^*(x)\equiv0$ iff $x \in \overline{H^-}$\\
\hline
\end{tabular}
\end{center}
\caption{\small Asymptotic behavior of $(S(x), I(x))$ as $d_S \to 0$  or $d_I \to 0$.}
\label{t2}
\end{table}
\begin{table}[!tbp]
\begin{center}
\begin{tabular}{|c||p{10cm}|}
\hline
Model  & Limit of $(S(x), I(x))$ as $d_S \to 0$ and $d_I \to 0$\\
\hline
MO & For the case $d_I/d_S\to d\in(0,\infty)$,  $S^*(x)>0$ and  $I^*(x) \geq 0$ but $I^*(x) \not\equiv 0$\\
\hline
MW & Unknown\\
\hline
SO & $S^*(x)> 0$ and $I^*(x)\equiv0$ iff $x \in \overline{H^-}$ when $d_I/d_S \to [0,\infty)$, $S^*(x)\geq 0$ and $I^*(x)\equiv 0 $ when $d_I/d_S\to\infty$\\
\hline
SW & Unknown\\
\hline
\end{tabular}
\end{center}
\caption{\small Asymptotic behavior of $(S(x), I(x))$ as both $d_S \to 0$ and $d_I \to 0$.}
\label{t3}
\end{table}

When the movement rate $d_S$ of the susceptibles tends to zero, the asymptotics of solutions have been well understood to a large extent as seen in Table \ref{t2}, and the asymptotic profiles of EE illustrated in Fig.\ref{fig1} are consistent with analytical results. It is worth mentioning that for MO, since the parameter values are taken so that $1=N<\int_0^1 \frac{\gamma(x)}{\beta(x)}dx$, we have the convergence $I\to 0$ as $d_S \to 0$ according to the results of \cite{Wu-Zou} which our numerical simulations fit well.

With the same parameter values as in Fig.\ref{fig1}, we illustrate the asymptotic profiles of EE as $d_I \to 0$ in Fig.\ref{fig2}. For the two standard incidence infection models SO and SW, our simulations show that the limiting profile of $I$ for both models is positive only at high-risk sites which match well with the analytical results. The limiting profile of $S$ for model SW is constant because of the special choice of $\Lambda$ (see \cite[Theorem 5.2]{LiPengWang}). For models MW and MO, the exact limiting behavior of $I(x)$ remains open except knowing that its total population is positive (see Table \ref{t2}). Our numerical simulations in Fig.\ref{fig2} demonstrate that the infectious disease tends to aggregate in a narrow region and is eradicated outside this region, where model MO has a narrower aggregation region than model MW.
{\color{black} We remark that in our simulation the condition $\gamma<N\beta$ required in \cite{LiBo1} is not satisfied on $[0,1]$, and we observe that $S$ tends to a positive constant though its rigorous proof still remains open.}

The asymptotic behavior of EE as $d_S \to0$ and $d_I \to 0$ is only partially understood (see results in Table \ref{t3}). The numerical simulations shown in Fig.\ref{fig3} verify the known results on models SO and MO where the asymptotic profiles of $(S,I)$ coincide because of our choice of the parameter values.
However, the asymptotic behavior of EE as $d_S \to0$ and $d_I \to 0$  for models MW and SW entirely remains open and our numerical simulations have the following predictions. Firstly, for MW, the simulation implies that $S^*(x)>0$ and $I^*(x)\geq 0$ but $I^*(x)\not \equiv 0$ as  $d_S \to0$ and $d_I \to 0$ with $d_I/d_S \to d \in (0, \infty)$, which is analogous to the asymptotic behavior of EE for MO. In other words, the birth-death effect seems to be not important for SIS models with mass-action infection mechanisms if both diffusion rates of the susceptible and infectious are small with the same order. Secondly, for model SW, the numerical simulation shows that $S^*(x)$ is a positive constant and $I^*(x)\geq 0$ where $I^*(x) \equiv 0$ if and only if $x \in \overline{H^-}$. These simulations suggest possible asymptotic behavior of models MW and SW as $d_S \to0$ and $d_I \to 0$ for further analytical pursues.

Finally, to see whether the inclusion of a moderate-risk region will affect the asymptotic profiles of EE as considered in \cite{Peng-Yi}, we choose appropriate functions for $\beta(x)$ and $\gamma(x)$ as
\begin{equation}\label{moderate}
\beta(x)=
\begin{cases}
1, & x\in [0,0.75]\\
2x-0.5,& x\in [0.75,1]
\end{cases},\
\gamma(x)=
\begin{cases}
-2x+1.5, & x\in [0,0.25]\\
1,& x\in [0.25,1]\\
\end{cases}
\end{equation}
 such that $\beta(x)= \gamma(x)$ on the interval $[0.25,0.75]$ (moderate-risk region), see a plot in Fig.\ref{fig0}(b), and perform numerical simulations with small $d_I$. For model
MW, Fig.\ref{fig4}(a) indicates that the infected population tends to aggregate on two narrow regions instead of one, compared to the case without a moderate-risk region as illustrated in Fig.\ref{fig2}.
Moreover, the simulation in Fig.\ref{fig4}(b) illustrates that the limiting profile of $I$ of SO, SW and MO
 is positive only at high-risk sites. This is in sharp contrast with Fig.\ref{fig2} where there is no moderate-risk region and the limiting profile of $I$ for model MO is positive only on a narrow part within the high-risk region.
\begin{figure}[htb]
\centering
\includegraphics[width=6cm]{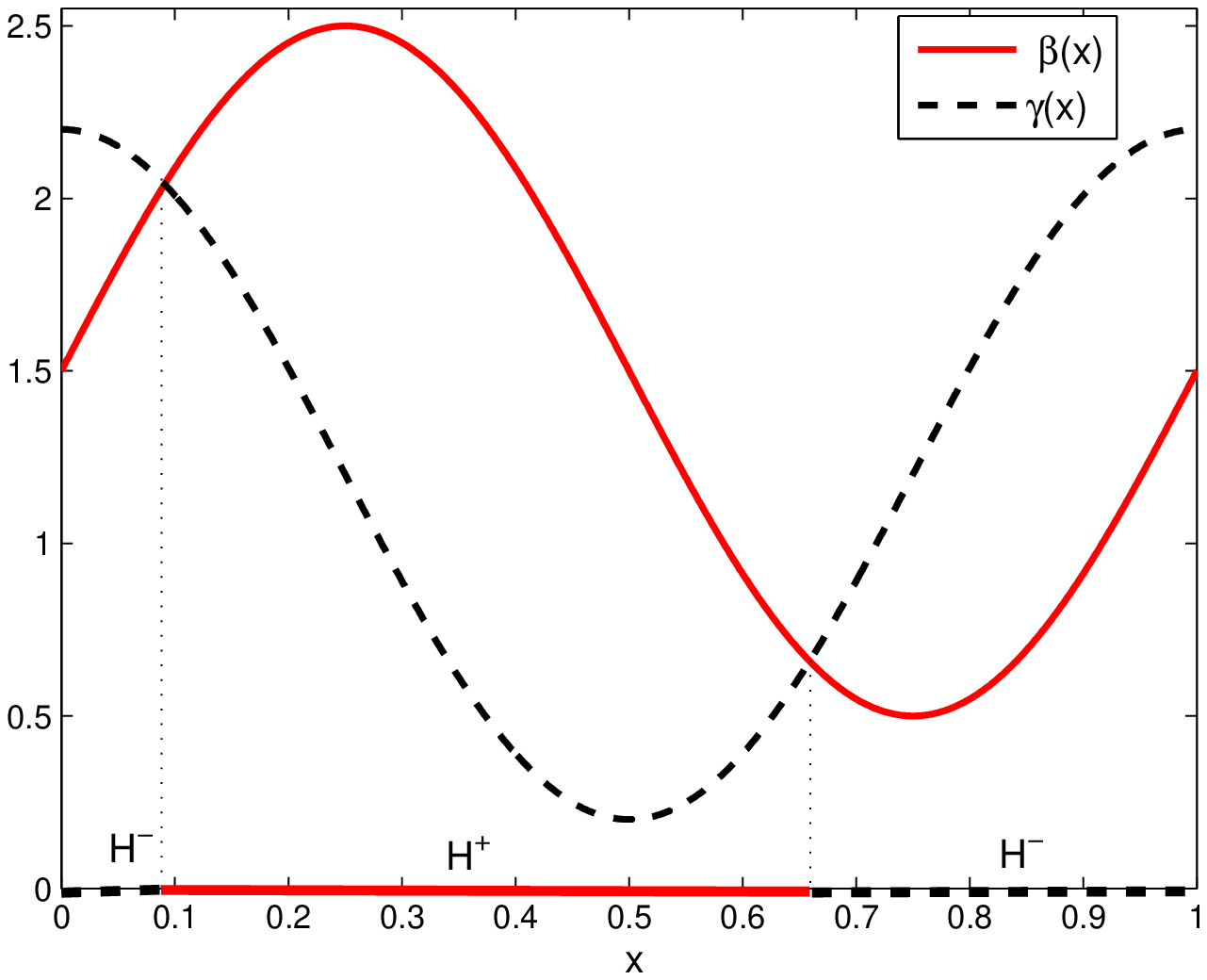}
\includegraphics[width=6cm]{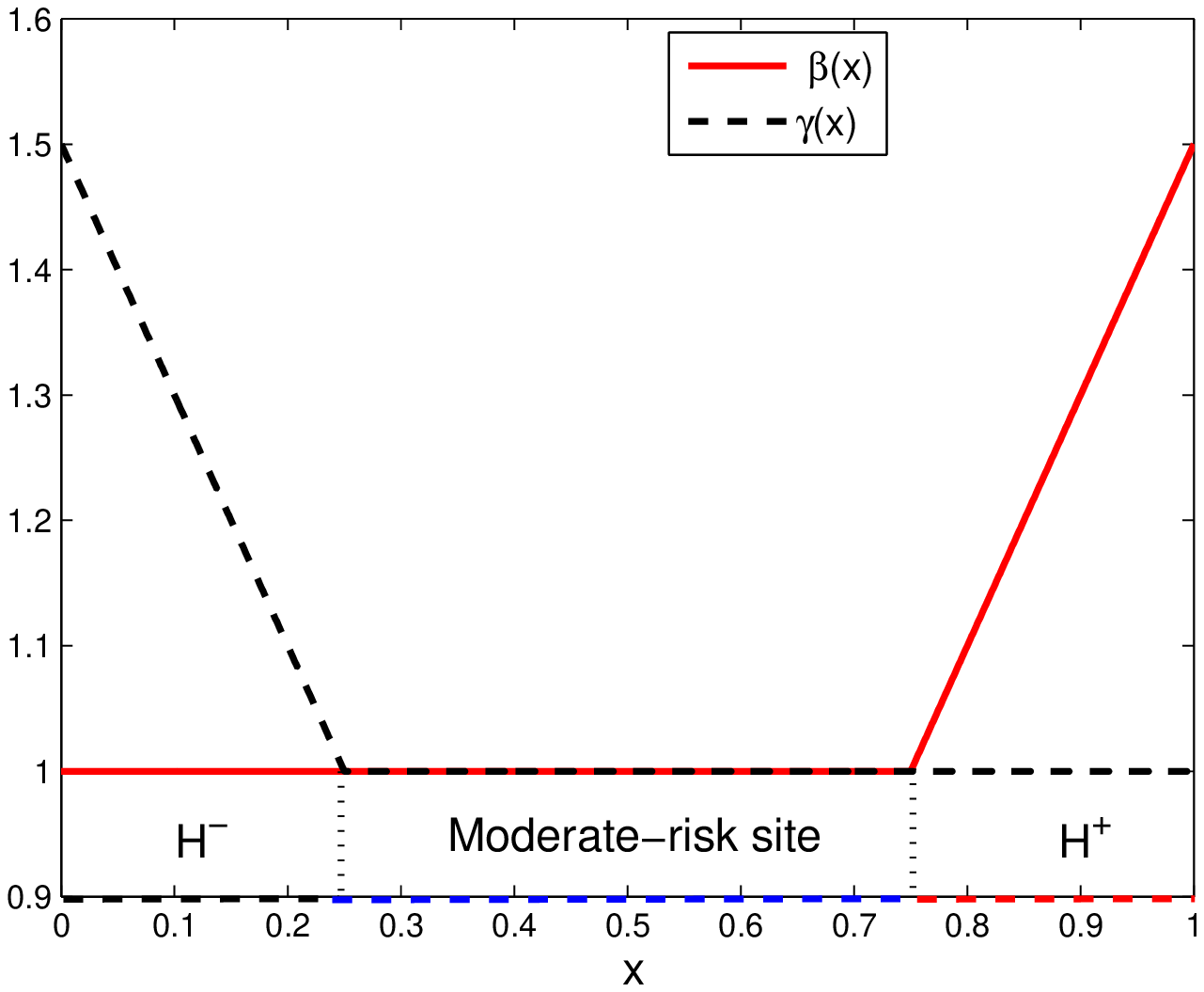}

(a) \hspace{5cm} (b)
\caption{\small (a) Graphs of $\beta(x)=1.5+\sin(2\pi x)$, $\gamma(x)=1.2+\cos(2\pi x)$ for $x\in [0,1]$, and the set $H^-$ and $H^+$ (reproduction of Fig.1 in \cite{Allen}); (b) Graphs of $\beta(x)$ and $\gamma(x)$ given by (\ref{moderate}) with a moderate-risk region in $[0,1]$.}
\label{fig0}
\end{figure}

\begin{figure}[htb]
\centering
\includegraphics[width=6cm, height=4.5cm]{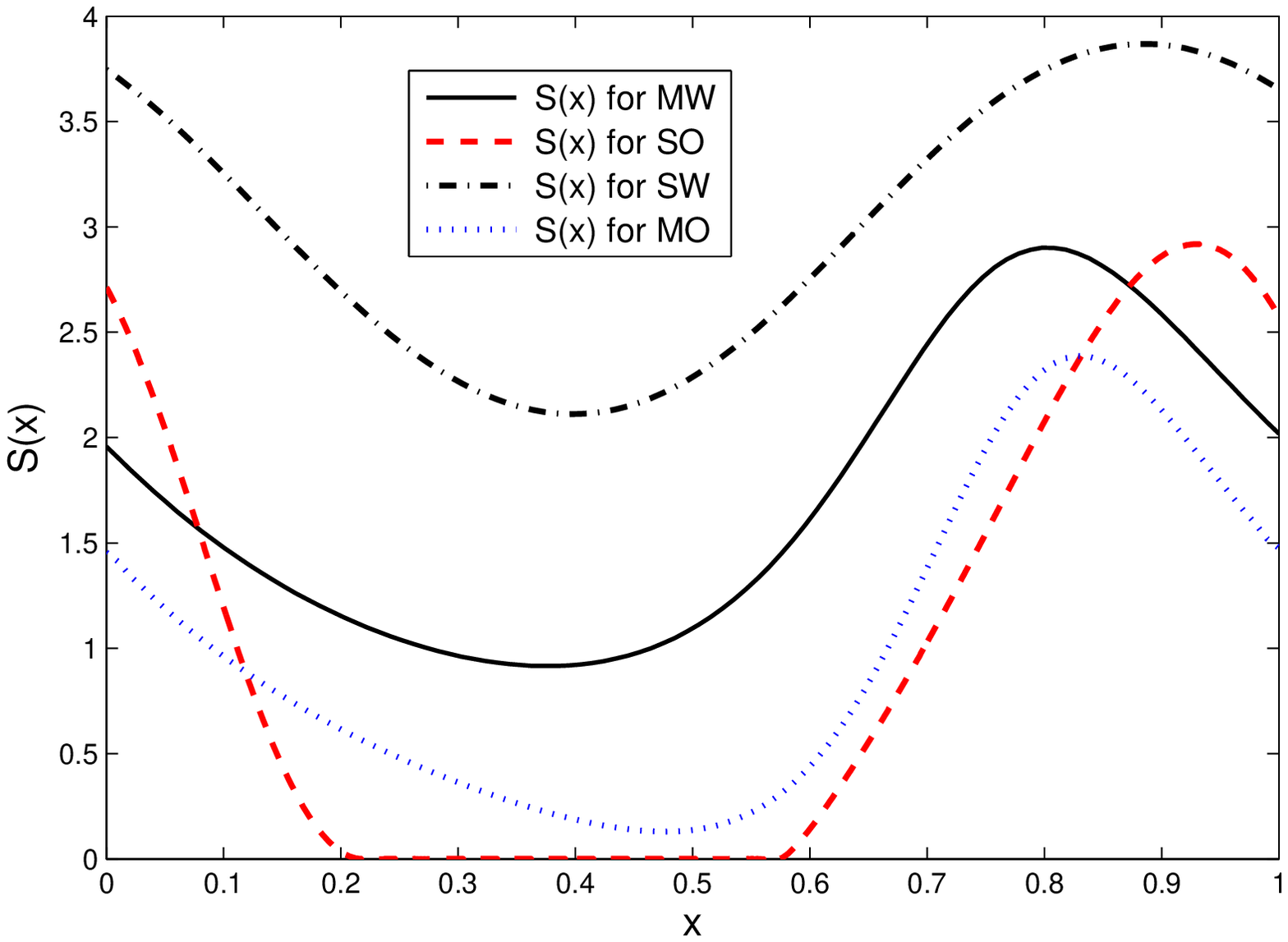}
\includegraphics[width=6cm, height=4.5cm]{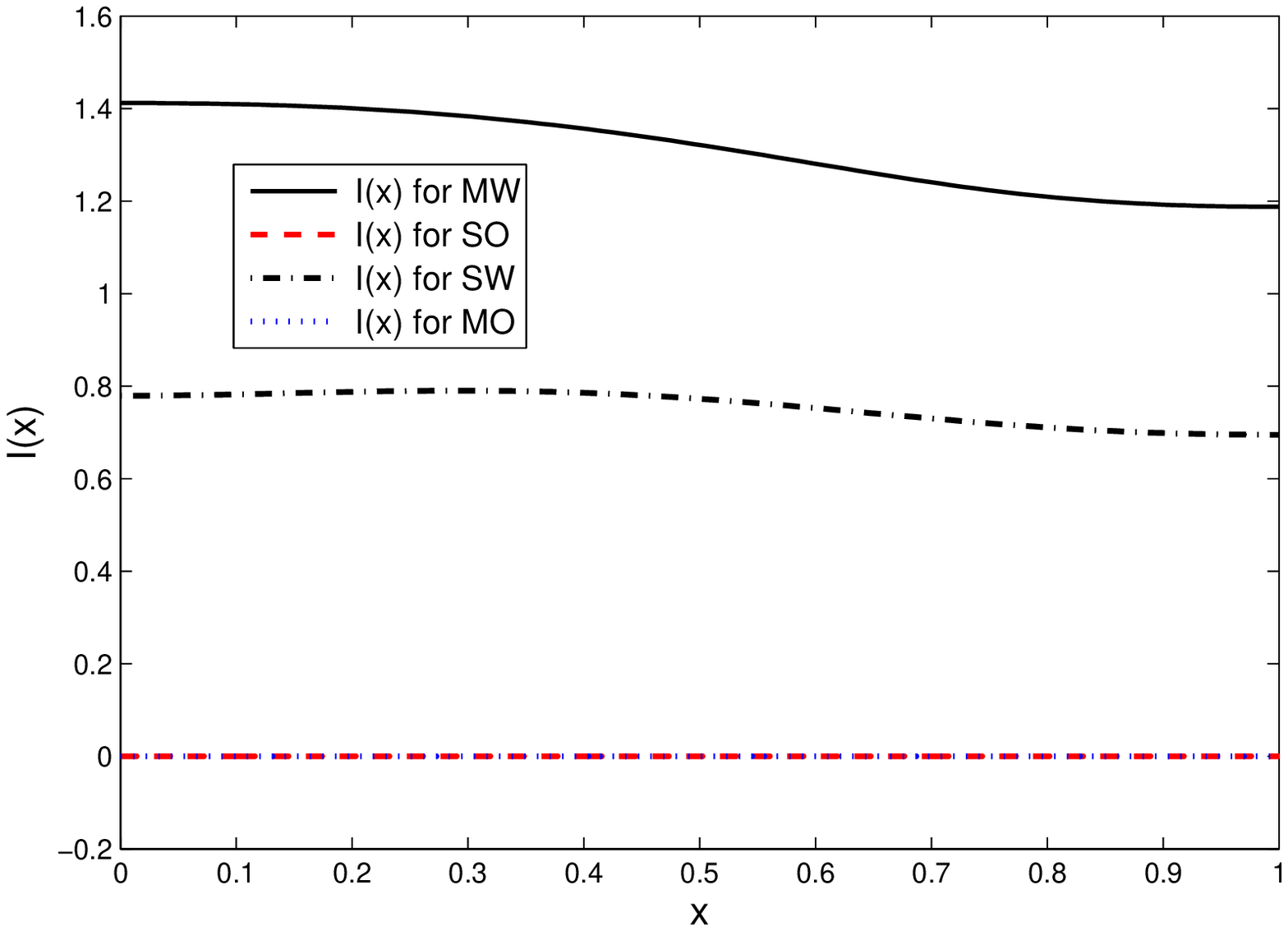}
\caption{\small Numerical simulations of the profile of $(S(x), I(x))$ as $d_S \to 0$ for systems MO, MW, SO and SW, where parameters are chosen as: $d_S=10^{-6}$, $d_I=1$, $\Lambda(x)=3$, $\mu(x)=0.5+x$ and $\beta(x)$ and $\gamma(x)$ are as plotted in Fig.\ref{fig0}(a). }
\label{fig1}
\end{figure}

\begin{figure}[htb]
\centering
\includegraphics[width=6cm, height=4.5cm]{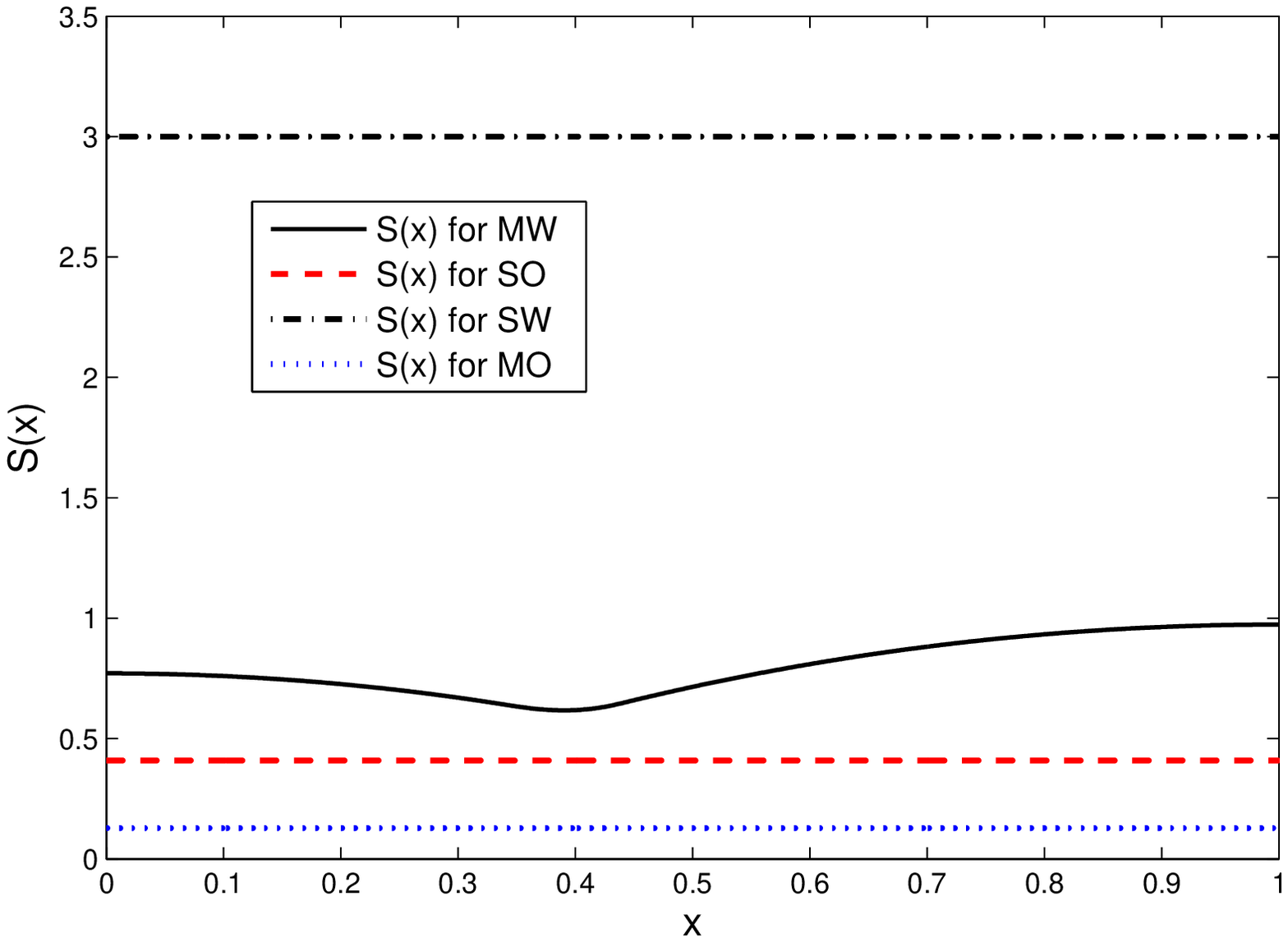}
\includegraphics[width=6cm, height=4.5cm]{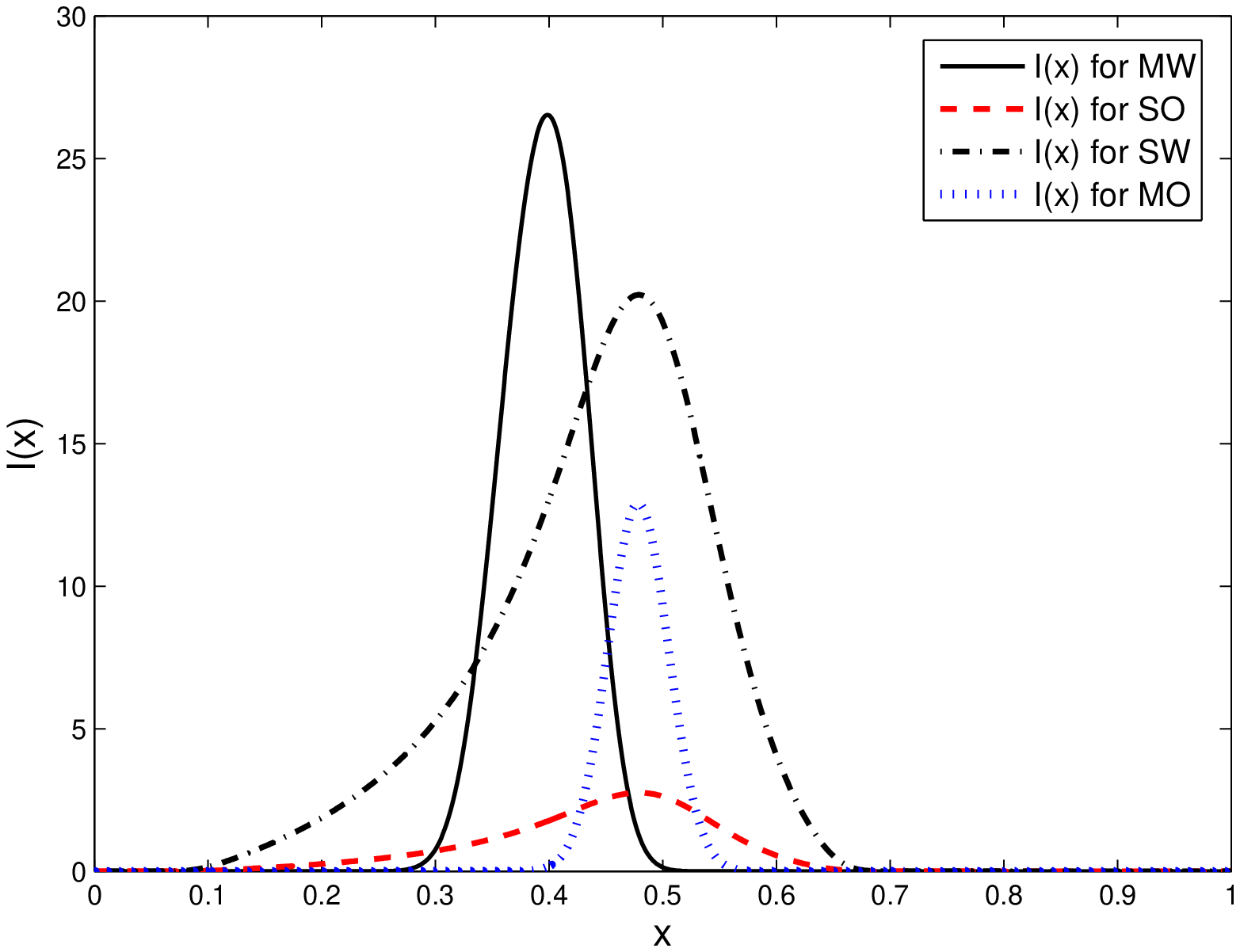}
\caption{\small Numerical simulations of the profile of $(S(x), I(x))$ as $d_I \to 0$ for systems MO, MW, SO and SW, where $d_S=1, d_I=10^{-5}$ and other parameters are chosen same as those in Fig.\ref{fig1}.}
\label{fig2}
\end{figure}

\begin{figure}[htb]
\centering
\includegraphics[width=6cm, height=4.5cm]{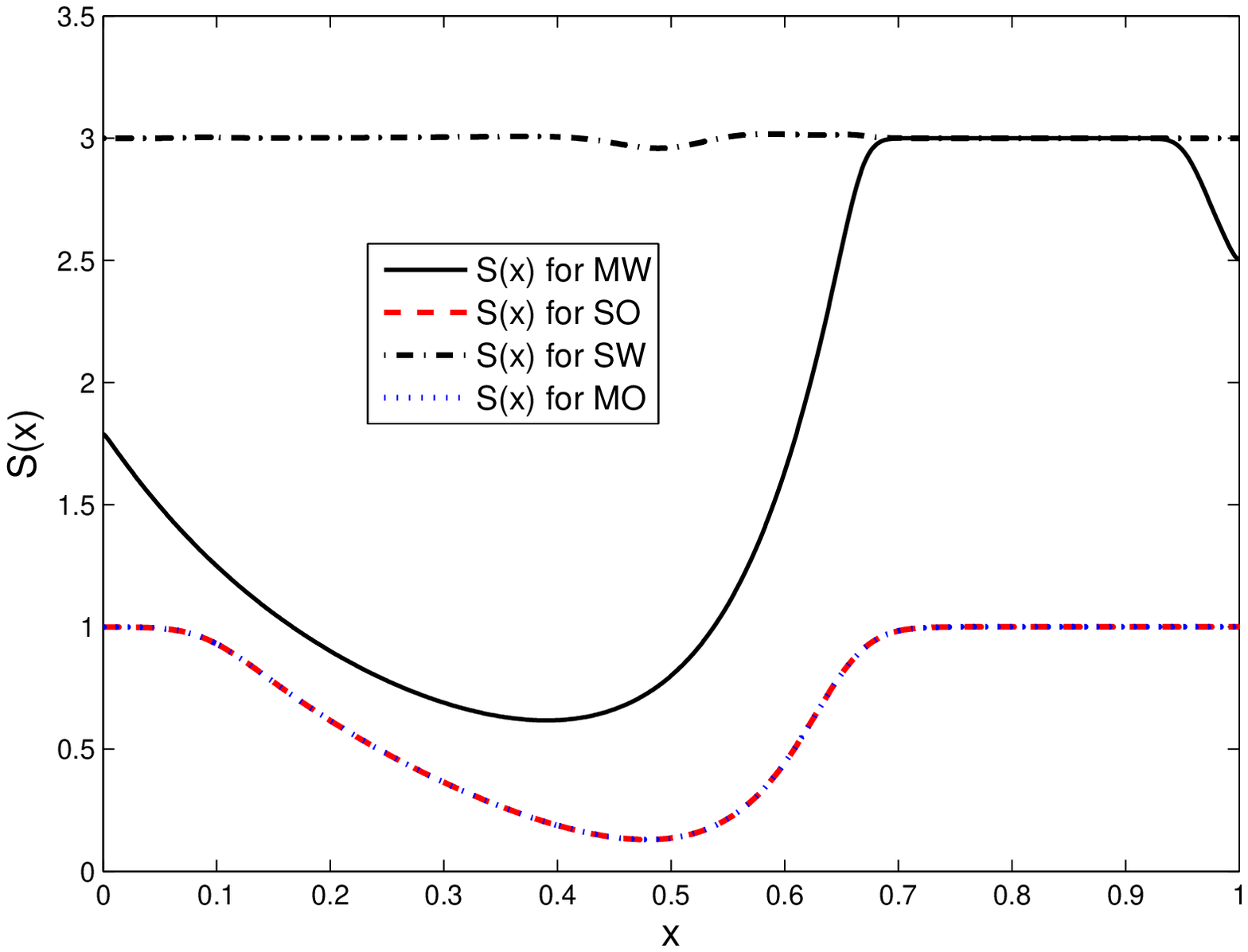}
\includegraphics[width=6cm, height=4.5cm]{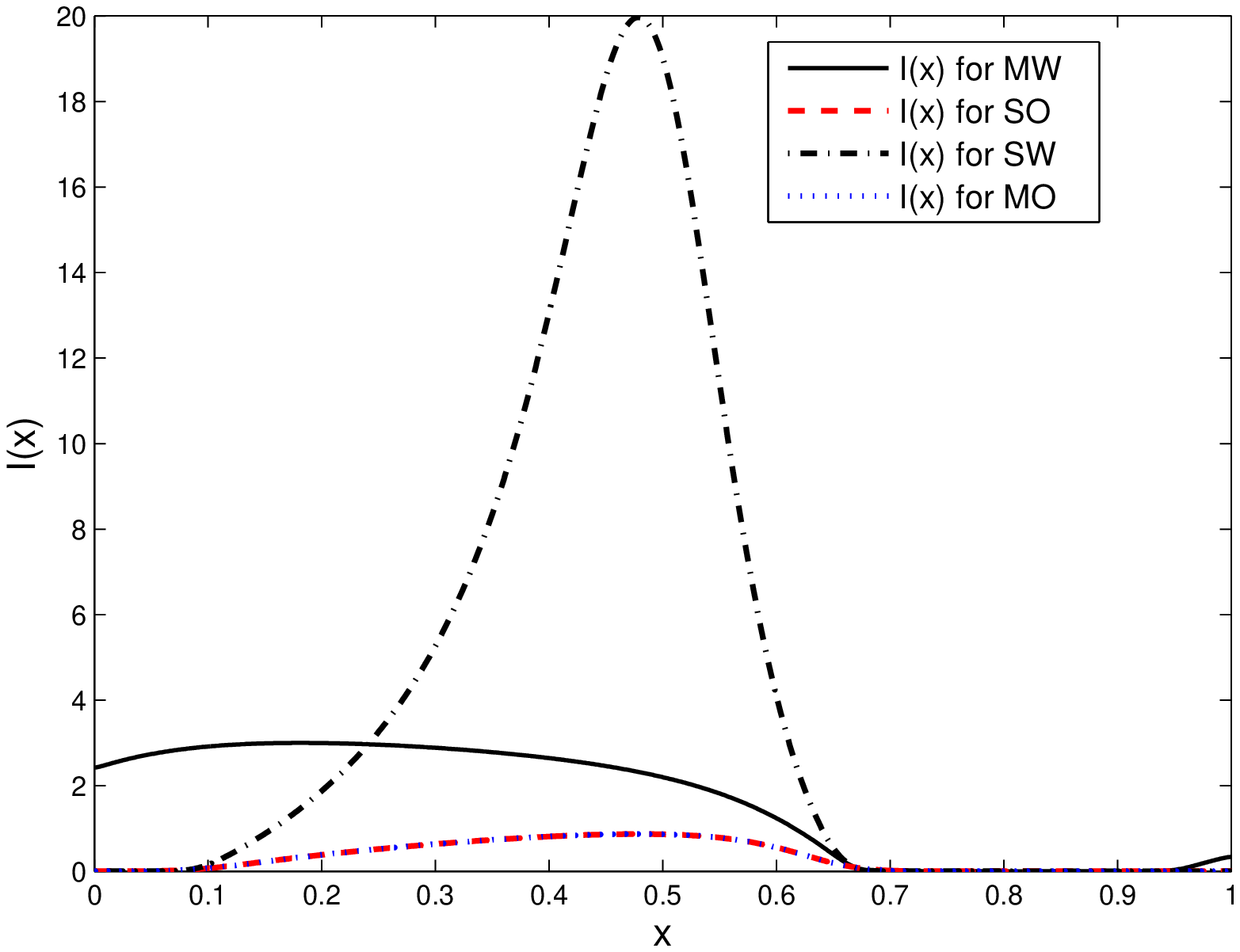}
\caption{\small Numerical simulations of the profile of $(S(x), I(x))$ as $d_S, d_I \to 0$ for systems MO, MW, SO and SW, where  $d_S=d_I=10^{-5}$ and other parameters are chosen same as those in Fig.\ref{fig1}.}
\label{fig3}
\end{figure}

\begin{figure}[htb]
\centering
\includegraphics[width=6cm,height=4.5cm]{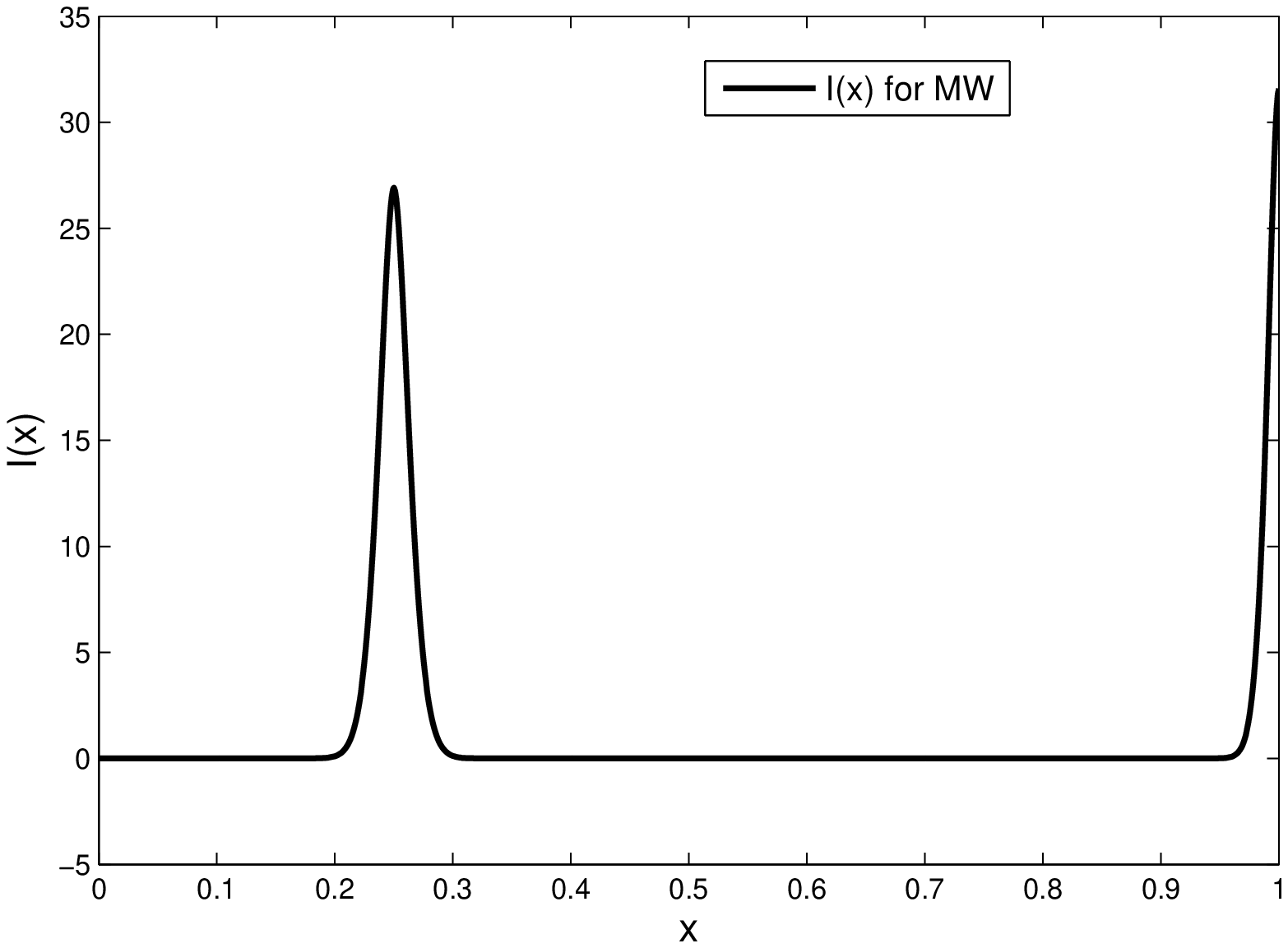}
\includegraphics[width=6cm,height=4.5cm]{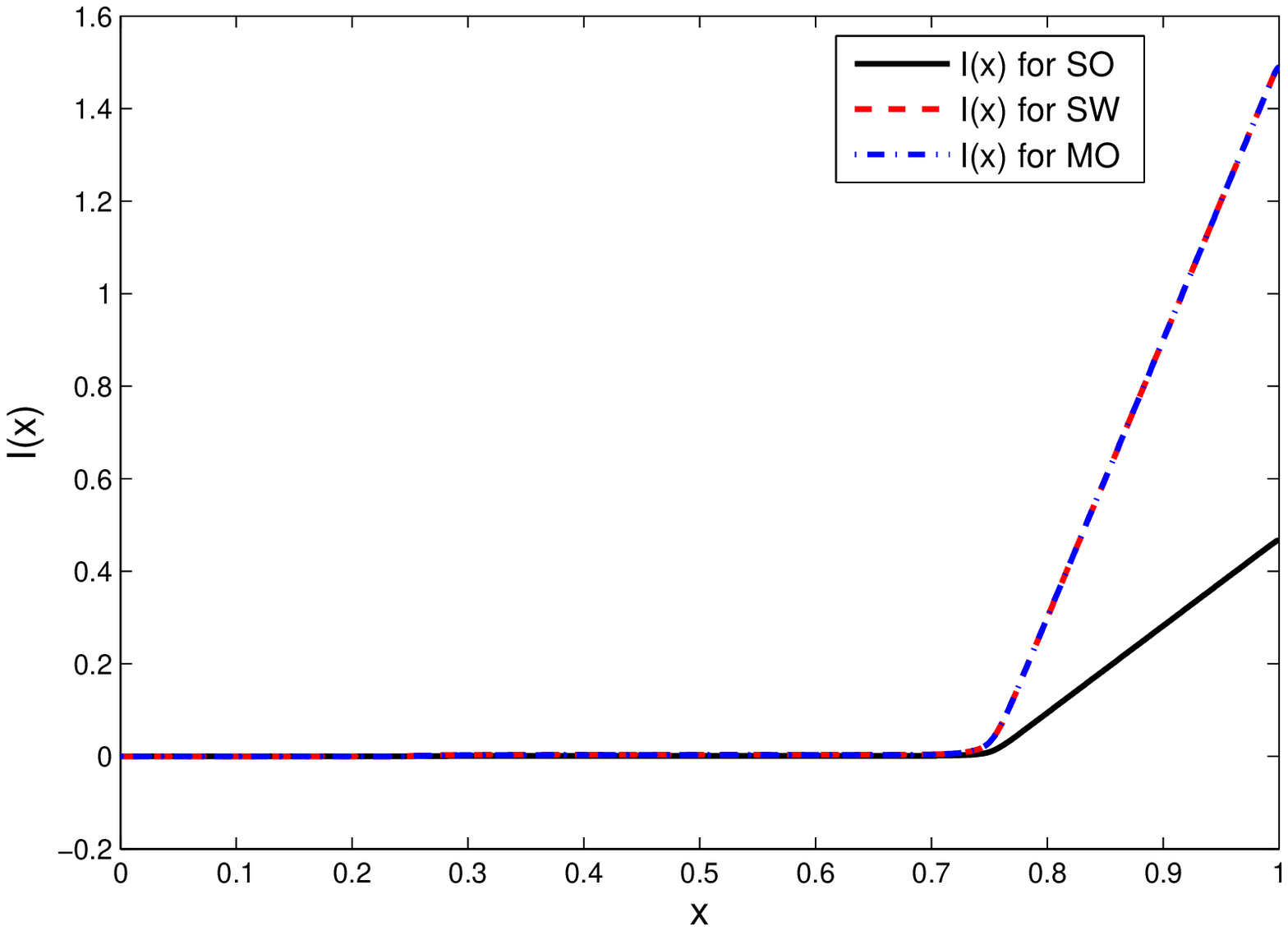}

(a) \hspace{5.5cm} (b)
\caption{\small Numerical simulations of the asymptotic profile of $I(x)$ as $d_I \to 0$ for systems MO, MW, SO and SW with a moderate-risk site, where $d_S=1,d_I=10^{-5}$, $\Lambda(x)=3$, $\mu(x)=0.5+x$, $\beta(x)$ and $\gamma(x)$ are given by (\ref{moderate}) as plotted in Fig.\ref{fig0}(b).}
\label{fig4}
\end{figure}

\subsubsection{Implication on disease control}
We now discuss numerous implications/comments on disease control based on analytical and numerical results summarized in the preceding subsections.

First consider models SO and MO which have conserved total population but subject to different infection mechanism. For model SO with any magnitude of total population, it is possible to eliminate the disease entirely by restricting $d_S$ while the disease cannot be eradicated on high-risk sites by limiting $d_I$ (see Table 2, Fig.\ref{fig1} and Fig.\ref{fig2}). As for model MO, restricting $d_S$ can eliminate the disease only if the total population is small (see Table \ref{t2}), whereas the infected individuals tend to aggregate on a narrow region if $d_I$ is small by the observation from  Fig.\ref{fig2}. Thus, if the total number of population remains unchanged, we may conclude that the disease described by standard incidence infection mechanism modeled by SO is easier to control by limiting the motility $d_S$ of susceptible population compared to the mass-action infection mechanism modeled by MO. Nevertheless, the disease subject to mass action infection mechanism can be eradicated to a larger extent (region) if the motility $d_I$ of infected individuals is restricted.

Now consider models MW and SW that have the same linear recruitment but different infection mechanisms. From Table \ref{t2}, Fig.\ref{fig2} and Fig.\ref{fig3}, we see that the infectious disease cannot be eliminated at all by restricting $d_S$ for either models due to the source term of susceptible population, while restricting $d_I$ can eliminate the disease partially for both models but standard incidence infection mechanism seems to be more efficient than the mass-action one.

Let us also consider the effect of linear recruitment on the same infection mechanism; that is, we compare model SO with SW, and MO with MW. Recall that restricting the motility of susceptible population ($d_S$ is small) yields the extinction of disease subject to standard incidence infection mechanism in SO, but this strategy fails  for SW with linear recruitment subject to the same infection mechanism. Similar results hold between models MO and MW, but only with small total population.
When $d_I$ is small, the infectious disease modeled by SO and SW is eradicated/persistent at the same region but the latter has a larger total mass, whereas the infectious disease modeled by MW is less condensed compared to its counterpart MO.
Thus, if $d_S$ is small, whichever the infection mechanism is, a varying total population tends to enhance the persistence of disease, while this enhancement induced by standard incidence infection mechanics is not as strong as mass action one does. Nevertheless, for small $d_I$, the disease subject to mass-action infection mechanism modelled by MO and MW seems to be less endemic since the infected population is more concentrated (see Fig.\ref{fig2}). 

If the environment is modified to include a moderate-risk region (see a graph in Fig.\ref{fig0}(b)), then we see that for small $d_I$, the disease modelled by SO, SW and MO can be eradicated precisely at low-risk and moderate-risk sites (see Fig.\ref{fig4}(b)). This exhibits quite different behavior than that of model MW for which the infected disease may also persist in low-risk or moderate-risk sites but also be eradicated in part of high-risk sites (see Fig.\ref{fig4}(a)).
Compared to the profiles shown in Fig.\ref{fig2} for the case of small $d_I$ without moderate-risk site, from the standing point of disease control, this essentially implies that at least for model MO it is perhaps not a sound strategy to create a moderate-risk domain in the environment and restrict the motility of infected population at the same time.

We also would like to mention that due to the conservative property of the total population, the steady state problem of SO can be reduced to a single local elliptic equation while that of MO can be reduced to a single nonlocal elliptic equation. Hence, this property makes the corresponding system easier to attack, compared to the case of varying total population. Moreover, it is exactly because of this property that one can consider the asymptotic profiles of the positive solution for small $d_I$ and $d_I/d_S \to d_0$ for some $d_0$, as in \cite{Peng, Wu-Zou}. This seems to be a rather challenging task for the steady state of models MW and SW due to lack of appropriate a prior estimates.

Finally, it is perhaps worth mentioning that one can also consider the effects of large motility rate of susceptible or infected population, as in \cite{LiBo2, LiPengWang, Peng}. In fact, one can easily follow the arguments there and conclude that when the motility of the susceptible population tends to infinity, the density of the susceptibles becomes positive and homogeneous and the density of the infected is also positive but inhomogeneous throughout the habitat; similar result holds if the movement rate of the infected population becomes large. Since these results are essentially the same as before and they indicate that large diffusion rate of the susceptibles or infected does not help to eradicate the disease, we do not present these results in this paper.

\end{document}